\documentclass{mcom-l-no-footnote}

\usepackage{amssymb}

\usepackage{graphicx}
\usepackage{float}

\usepackage{bm}
\usepackage[ruled]{algorithm2e}

\newcommand{\field}[1]{\mathbb{#1}}
\newcommand{\R}{\field{R}}

\newcommand{\Z}{\field{Z}}

\newcommand\scalemath[2]{\scalebox{#1}{\mbox{\ensuremath{\displaystyle #2}}}}

\newtheorem{theorem}{Theorem}[section]

\theoremstyle{definition}

\theoremstyle{remark}

\numberwithin{equation}{section}

\begin{document}

\title[BFECC for linear hyperbolic systems w. appl. to Maxwell's equations]{Back and Forth Error Compensation and Correction Method for Linear Hyperbolic Systems With Application to the Maxwell's equations}


\author[Xin Wang]{Xin Wang}
\address{School of Mathematics, Georgia Institute of Technology,
Atlanta, GA 30332}
\curraddr{}
\email{xwang320@math.gatech.edu}

\author[Yingjie Liu]{Yingjie Liu}
\address{School of Mathematics, Georgia Institute of Technology,
Atlanta, GA 30332}
\curraddr{}
\email{yingjie@math.gatech.edu}


\date{}

\dedicatory{}


\begin{abstract}
We study the Back and Forth Error Compensation and Correction (BFECC) method for linear hyperbolic PDE systems. The BFECC method has been applied to schemes for advection equations to improve their stability and order of accuracy. Similar results are established in this paper for schemes for linear hyperbolic PDE systems with constant coefficients. We apply the BFECC method to central difference scheme and Lax-Friedrichs scheme for the Maxwell's equations and obtain second order accurate schemes with larger CFL number than the classical Yee scheme. The method is further applied to schemes on non-orthogonal unstructured grids. The new BFECC schemes for the Maxwell's equations operate on a single non-staggered grid and are simple to implement on unstructured grids. Numerical examples are given to demonstrate the effectiveness of the new schemes. 
\end{abstract}

\maketitle\thispagestyle{empty}


\section{Introduction}
The goal of this paper is to study finite difference schemes for the Maxwell's equations that are based on the back and forth error compensation and correction (BFECC) method \cite{dupont2003back}. Extensive studies have been done on finite difference time domain (FDTD) schemes for the Maxwell's equation \cite{taflove2005computational}. Compared with other methods, for example finite element schemes, FDTD methods are very efficient, easy to implement, and are able to model behaviors over all frequencies simultaneously \cite{taflove2005computational}. The classical Yee scheme \cite{yee1966numerical} is originally designed for uniform orthogonal grids. For non-uniform orthogonal grids, Yee scheme is known to be second order globally (though the local truncation error is first order) \cite{monk1994convergence, monk1994error}. It can also be generalized for irregular nonorthogonal grids, such as the Nonorthogonal FDTD scheme \cite{palandech1992modeling}, the Generalized Yee scheme \cite{gedneygeneralized} and the Overlapping Yee scheme \cite{liu2009overlapping}. These schemes require generation of (nonorthogonal or unstructured) staggered grids for $\bm{E}$ and $\bm{H}$ and the formulation and implementation on the unstructured staggered grids can be complicated. In this paper, we propose a simple finite difference scheme based on the BFECC method that requires very few modifications when changing from uniform non-staggered grids to unstructured non-staggered grids. 

Back and Forth Error Compensation and Correction (BFECC) method is introduced in \cite{dupont2003back, dupont2007back} to obtain a higher order scheme based on a lower order scheme for advection equations. Given a scheme for advection equations, the idea of BFECC method is to improve its accuracy by estimating using forward and then backward advections and correcting its leading order error. Suppose $\mathcal{L}$ is a $r$-th order linear scheme for scalar linear advection equations, where $r$ is an odd integer, the in general the BFECC scheme based on $\mathcal{L}$ is $(r+1)$-th order accurate, and is stable as long as scheme $\mathcal{L}$ has an amplification factor no more than $2$, thus has a larger CFL number than $\mathcal{L}$ \cite{dupont2003back, dupont2007back}. In this paper, we extend the BFECC method to linear hyperbolic systems, and show that similar accuracy and stability improvement can be achieved.

The BFECC method has been applied to level set interface computation and fluid simulations \cite{dupont2003back, dupont2007back, kim2007simulation, kim2005flowfixer, kim2007advections}. A two-step unconditionally stable MacCormack scheme and its generalization are developed in \cite{selle2008unconditionally} for fluid simulations. The property that BFECC stabilizes even an unstable scheme (with its amplification factor no more than $2$) is very helpful for systems because one doesn't have to compute the local characteristic information for constructing a low diffusion stable scheme. With the new extension to linear hyperbolic systems, we propose BFECC schemes for the Maxwell's equations which are second order accurate, easy to implement, and have larger CFL numbers than that of the classic Yee scheme \cite{yee1966numerical}. Given the accuracy improving ability of the BFECC method, we propose to use a simple first order scheme that is based on the least square local linear approximation as the underlying scheme for BFECC on unstructured grids which is very easy to implement after being stabilized. Numerical examples show that the scheme remains to be second order on non-orthogonal grids.

The rest of the paper is organized as follows. In Section \ref{sec:bfecc_system}, we discuss the BFECC method for linear hyperbolic PDE systems and prove the stability and accuracy theorems. In Section \ref{sec:bfecc_maxwell}, we apply BFECC to the Maxwell's equations. On uniform orthogonal grids, we use central difference and Lax-Friedrichs schemes as the underlying schemes for the BFECC method. Order of accuracy and CFL numbers for the corresponding schemes are discussed. On unstructured grids, we present a first order scheme based on the least square local linear approximation and use it as the underlying scheme. The divergence of the magnetic field and the perfectly matched layer \cite{berenger1994perfectly} implementation are also discussed. Numerical examples are presented in Section \ref{sec:num_examples}. We conclude the paper in Section \ref{sec:conclusion}. A detailed error analysis of the BFECC applied to the central difference scheme is presented in the appendix. 

\section{BFECC method for homogeneous linear hyperbolic PDE systems with constant coefficients}\label{sec:bfecc_system}

In this section, we discuss the BFECC method for homogeneous linear hyperbolic PDE systems with constant coefficients. Denote $\bm{u}(\bm{x}, t)$ the vector of unknown functions, where $\bm{x} = (x_1, x_2, ..., x_d)^T \in \R^d$ and $t \in \R$ are the spatial and temporal variables. Consider a homogeneous linear hyperbolic PDE system with constant coefficients in the following form:
\begin{align}\label{eq:hyperbolic_system}
\partial_t \bm{u} + \sum_{i = 1}^{d} A_i \partial_{x_i} \bm{u} = 0,
\end{align}
where $A_i, i = 1, 2, ..., d$ are real constant matrices, and any linear combination $\sum_{i = 1}^{d} \alpha_i A_i$ is diagonalizable with real eigenvalues. When all the coefficient matrices $A_i$ are symmetric, we say it is a symmetric linear hyperbolic system.  

We solve this system numerically with a finite difference scheme. For simplicity of discussion, we assume a uniform orthogonal grid is used and discuss the scheme in the whole space. Denote the mesh sizes
\begin{equation*}
\Delta \bm{x} = ( \Delta x_1, \Delta x_2, ..., \Delta x_d ),
\end{equation*}
and $\Delta t_n = t_{n+1} - t_{n}$ (we omit subscript $n$ when $\Delta t_n$ is the same for all $n$). Denote the numerical solution
\begin{equation*}
\bm{U}_{\bm{j}}^{n} \approx \bm{u}(j_1 \Delta x_1, j_2 \Delta x_2, ..., j_d \Delta x_d, t_n),
\end{equation*}
where $\bm{j} = (j_1, j_2, ..., j_n)$ is the multi-index vector. Denote $\bm{U}^n = \left\lbrace \bm{U}^{n}_{\bm{j}}: \forall \bm{j} \right\rbrace$ the collection of numerical solution at all grid points at the time $t_n$.

Suppose $\mathcal{L}$ is a numerical scheme for this system, i.e. 
\begin{equation*}
\bm{U}^{n+1} = \mathcal{L} \bm{U}^{n}.
\end{equation*}
In this paper, all the schemes we discussed are linear schemes, i.e. $\mathcal{L}$ is a linear operator.

We define $\mathcal{L}^{*}$ the backward update step from $t_{n+1}$ to $t_{n}$ by applying $\mathcal{L}$ to the time-reversed system:
\begin{align*}
\partial_t \bm{u} - \sum_{i = 1}^{d} A_i \partial_{x_i} \bm{u} = 0.
\end{align*}
 
By applying the Back and Forth Error Compensation and Correction (BFECC) steps \cite{dupont2003back, dupont2007back}, we obtain a new scheme $\mathcal{L}_{BFECC}$ which updates the solution in three steps:

\begin{enumerate}
\item{\bf Solve forward}. \\
$\tilde{\bm{U}}^{n+1} = \mathcal{L} \bm{U}^{n}$.
\item{\bf Solve backward}. \\
$\tilde{\bm{U}}^{n} = \mathcal{L}^{*} \tilde{\bm{U}}^{n+1}$.
\item{\bf Solve forward with the modified solution at time $t_n$}. \\
$\bm{U}^{n+1} = \mathcal{L} \left( \bm{U}^n + \bm{e}^{(1)} \right)$, where $\bm{e}^{(1)} = \frac{1}{2} \left( \bm{U}^n - \tilde{\bm{U}}^n \right)$.
\end{enumerate}

$\bm{U}^{n}$ and $\tilde{\bm{U}}^{n}$ should have been the same if there were no numerical error. Therefore $e^{(1)}$ provides an estimate of the value lost during the forward step, which is then compensated to $\bm{U}^{n}$ before performing the final forward step. In general, for linear advection equations, BFECC can improve the order of accuracy by one for odd order schemes and also improve stabilities of the schemes (see \cite{dupont2003back, dupont2007back}). We establish similar results for systems of equations in the following theorems with the help of techniques in \cite{zhang2005analysis, dupont2007back}.

In the following discussion, we consider system (\ref{eq:hyperbolic_system}) in $\prod_{i=1}^{d} [0, 1]$ with periodic boundary conditions. And we assume the numerical scheme $\mathcal{L}$ is a linear scheme. Let $\Delta x_j = \frac{1}{N_j}$ for $j = 1, 2, ..., d$. The numerical solutions are then defined at any time on $\mathcal{D}_{\bm N} = \Z^d \cap \prod_{i=1}^{d} [0, N_j - 1]$, where $\bm{N} = (N_1, N_2, ..., N_d)$.  Let $\mathcal{F}_{\bm{N}} = \Z^d \cap \prod_{i=1}^{d} [1 - N_j, N_j - 1]$ be the set for the dual indices of the finite Fourier series. Expand $\bm{U}^n$ as a finite Fourier series 
\begin{equation*}
\bm{U}_{\bm{j}}^{n} = \sum_{\bm{k} \in \mathcal{F}_{\bm{N}}} \bm{C}_{\bm{k}}^n e^{2 \pi i \bm{k} \cdot \bm{x}_{\bm{j}}},
\end{equation*}
where $\bm{j} \in \mathcal{D}_{\bm{N}}$ and $\bm{x}_{\bm{j}} = (j_1 \Delta x_1 , j_2 \Delta x_2, ..., j_d \Delta x_d )$.

Since scheme $\mathcal{L}$ is a linear scheme, the  coefficients of the Fourier series get updated as
\begin{align*}
\bm{C}_{\bm{k}}^{n+1} = Q_{\mathcal{L}}(\bm{k}) \bm{C}_{\bm{k}}^{n},
\end{align*}
where $Q_{\mathcal{L}}(\bm{k})$ is the Fourier symbol matrix for $\mathcal{L}$. 

{\bf Remark} Note that scheme $L$ is $l^2$ stable if the spectral radius $\rho(Q_{\mathcal{L}}(\bm{k})) < 1$ for all $\bm{k} \in \mathcal{F}_{\bm{N}}$ or $Q_{\mathcal{L}}(\bm{k})$ is diagonalizable and $\rho(Q_{\mathcal{L}}(\bm{k})) \leq 1$ for all $\bm{k} \in \mathcal{F}_{\bm{N}}$.

Denote $Q_{\mathcal{L}^{*}}(\bm{k})$ the Fourier symbol matrix of $\mathcal{L}^{*}$. Then Fourier symbol matrix $Q_B$ for the BFECC scheme based on $\mathcal{L}$ is
\begin{equation*}
Q_B = Q_{\mathcal{L}} \left( I + \frac{1}{2} ( I - Q_{\mathcal{L}^{*}} Q_{\mathcal{L}}) \right).
\end{equation*}

\subsection{Stability}
In general, BFECC method improves the stability of an underlying scheme $\mathcal{L}$ for the scalar hyperbolic equation $u_t + \bm{v} \cdot \nabla u = 0$ \cite{dupont2003back, dupont2007back}. It increases the CFL numbers of conditionally stable schemes (for example, the upwind scheme) and makes unstable schemes (for example, the central difference scheme) conditionally stable. We generalize this property of BFECC method for linear hyperbolic systems with constant coefficients. The result is summarized in the following theorem.

\begin{theorem}\label{thm:stability_system}
Let $\mathcal{L}$ be a linear scheme for system \ref{eq:hyperbolic_system}. Suppose $Q_{\mathcal{L}}$ and $Q_{\mathcal{L}^{*}}$ satisfies the following conditions

\begin{itemize}
\item[1] $Q_{\mathcal{L}^{*}}(\bm{k}) = \overline{Q_{\mathcal{L}}(\bm{k})}$ for all $\bm{k} \in \mathcal{F}_{\bm{N}}$, where $\overline{Q_{\mathcal{L}}(\bm{k})}$ is its complex conjugate, and  

\item[2] $Q_{\mathcal{L}^{*}}(\bm{k}) Q_{\mathcal{L}}(\bm{k}) = Q_{\mathcal{L}}(\bm{k}) Q_{\mathcal{L}^{*}}(\bm{k})$ for all $\bm{k} \in \mathcal{F}_{\bm{N}}$, and 

\item[3] $\rm{Re} (Q_{\mathcal{L}}(\bm{k}))$ and $\rm{Im} (Q_{\mathcal{L}}(\bm{k}))$ are diagonalizable with real eigenvalues for all all $\bm{k} \in \mathcal{F}_{\bm{N}}$. 
\end{itemize}

Then 
$|\rho(Q_B(\bm{k}))| \leq 1$ for all $\bm{k} \in \mathcal{F}_{\bm{N}}$ if and only if  $|\rho(Q_{\mathcal{L}}(\bm{k}))| \leq 2$ for all $\bm{k} \in \mathcal{F}_{\bm{N}}$.
\end{theorem}

\begin{proof}
	We first show that $\lambda_j(Q_B) = \left( 1 + \frac{1}{2} ( 1- |\lambda_j(Q_{\mathcal{L}})|^2) \right) \lambda_j(Q_{\mathcal{L}})$ under the assumptions in the theorem, where $\lambda_j(Q_{\mathcal{L}})$ and $\lambda_j(Q_B)$ are eigenvalues of $Q_{\mathcal{L}}$ and $Q_B$, respectively, $j = 1, 2, ..., d$ 
	
	Let $X = \rm{Re}(Q_{\mathcal{L}})$ and $Y = \rm{Im}(Q_{\mathcal{L}})$. Since $\bar{Q}_{\mathcal{L}} Q_{\mathcal{L}} = Q_{\mathcal{L}} \bar{Q}_{\mathcal{L}}$, we have
	\begin{equation*}
	(X - iY)(X + iY) = (X + iY)(X - iY) \Rightarrow X Y = Y X.
	\end{equation*}
	Since $X$ and $Y$ are diagonalizable with real eigenvalues and they commute, there is a basis set of real eigenvectors $\{v_j\}_{j = 1, 2, ..., n}$ that diagonalizes $X$ and $Y$ simultaneously. Then $v_i$'s are also eigenvectors of $Q_{\mathcal{L}}$ and $\bar{Q}_{\mathcal{L}}$, and the corresponding eigenvalues are complex conjugate of each other, i.e. $\lambda_j(\bar{Q}_{\mathcal{L}}) = \bar{\lambda}_j(Q_{\mathcal{L}})$ for $j = 1, 2, ..., d$. 
	
	By the assumption $Q_{\mathcal{L}^{*}} = \bar{Q}_{\mathcal{L}}$, we get $Q_B = Q_{\mathcal{L}} \left( I + \frac{1}{2} ( I - \bar{Q}_{\mathcal{L}} Q_{\mathcal{L}}) \right)$, and thus 
	\begin{equation*}
	\lambda_j(Q_B) = \left( 1 + \frac{1}{2} ( 1- |\lambda_j(Q_{\mathcal{L}})|^2) \right) \lambda_j(Q_{\mathcal{L}})
	\end{equation*}
	for $j = 1, 2, ..., d$.

Let $\zeta = |\lambda_j(Q_{\mathcal{L}})|$.By studying the function $f(\zeta) = | 1 + \frac{1}{2} ( 1 - \zeta^2) | \zeta$ for $\zeta \in [0, \infty)$, we see that $|f(\zeta)| \leq 1$ if and only if $\zeta \leq 2$, i.e. $|\lambda_j(Q_B)| \leq 1$ if and only if $|\lambda_j(Q_{\mathcal{L}})| \leq 2$, therefore the conclusion of the theorem follows.
\end{proof}

{\bf Remarks}
\begin{itemize}
\item[1.] Under the assumption of the theorem, Fourier symbol matrix $Q_B$ has a complete (real) eigenvector basis, so $| \rho(Q_B) | \leq 1$ implies $l^2$ stability.

\item[2.] Condition 1 follows the same assumption in the BFECC method for advection equations \cite{dupont2007back}, condition 2 requires that the scheme treats backward temporal direction the same as forward temporal direction, and condition 3 usually follows from the diagonalizability of coefficient matrix of the system. In particular, these assumptions on $Q_{\mathcal{L}}$ and $Q_{\mathcal{L}^*}$ are satisfied for several classical schemes. For example, consider the following one dimensional hyberbolic system
\begin{equation*}
\partial_t \bm{u} + A \partial_x \bm{u} = 0,
\end{equation*}
where $A$ is diagonalizable with real eigenvalues. 

Let $\mathcal{L}$ represent the central difference scheme for this system, i.e. 
\begin{align*}
\frac{\bm{U}^{n+1}_j - \bm{U}^{n}_j}{\Delta t} + A \frac{\bm{U}^{n}_{j+1} - \bm{U}^{n}_j}{2 \Delta x} = 0.
\end{align*}
Let $\lambda = \Delta t / \Delta x$, then
\begin{equation*}
Q_{\mathcal{L}}(k) = I - i \lambda \sin(2 \pi k h) A \, \text{ and } Q_{\mathcal{L}^{*}}(k) = I + i \lambda \sin(2 \pi k h) A.
\end{equation*}
Here $h = \Delta x$. We will continue to denote the spatial mesh size by $h$ when there is no ambiguity. 

Let $\mathcal{M}$ represents the Lax-Friedrichs scheme for this system, then
\begin{equation*}
Q_\mathcal{M}(k) = \cos(2 \pi k h) I - i \lambda \sin(2 \pi k h) A,
\end{equation*}
and 
\begin{equation*}
Q_{\mathcal{M}^{*}}(k) =  \cos(2 \pi k h)I + i \lambda \sin(2 \pi k h) A.
\end{equation*}
It is easy to see both schemes satisfy the assumptions of the theorem. 
  
\item[3.] A easier-to-check (but more restrictive) alternative for condition 3 in the theorem is to require $Q_{\mathcal{L}}$ being complex symmetric. This implies $X$ and $Y$ are real symmetric matrices, so they are diagonalizable with real eigenvalues. We will show in Section~\ref{sec:bfecc_maxwell} that this condition is satisfied for the central difference scheme and Lax-Friedrichs scheme for the Maxwell's equations.  
\end{itemize}

\subsection{Accuracy}
In general, BFECC method improves the accuracy of odd order schemes for advection equations \cite{dupont2003back, dupont2007back}. We extend this result to linear hyperbolic PDE systems with constant coefficients. 

Expand the solution into Fourier series
\begin{equation*}
\bm{u}(t, \bm{x}) = \sum_{\bm{k} \in \mathbb{Z}^d} \bm{C}_{\bm{k}}(t) e^{2 \pi i \bm{k} \cdot \bm{x}}
\end{equation*}

and plug in system (\ref{eq:hyperbolic_system}) to obtain
\begin{align*}
& \frac{\partial }{\partial t} \bm{C}_{\bm{k}}(t) = \left( - 2 \pi i \sum_{j = 1}^{d} k_j A_j \right) \bm{C}_{\bm{k}}(t) = P(i \bm{k}) \bm{C}_{\bm{k}}(t),
\end{align*}
where $P(i \bm{k})$ is a matrix with entries that are homogeneous linear polynomials in $i \bm{k}$ with real coefficients, $\bm{k} = (k_1, k_2, ..., k_d)^T$. Therefore
\begin{align*}
\bm{C}_{\bm{k}}(t + \Delta t) = e^{\Delta t P(i \bm{k})} \bm{C}_{\bm{k}}(t).
\end{align*}

Assume $\Delta x_1 = \Delta x_2 = ... = \Delta x_d = h$, and fix $\Delta t / h$ during the mesh refinement. We first quote a theorem of Lax \cite{lax1961stability}, 

\begin{theorem}\label{thm:lax}
For the linear hyperbolic PDE system (\ref{eq:hyperbolic_system}) with constant coefficients, a scheme $\mathcal{L}$ is $r$-th order accurate if and only if its Fourier symbol matrix $Q_{\mathcal{L}}$ satisfies
\begin{equation*}
Q_{\mathcal{L}} (\bm{k}) = e^{\Delta t P(i \bm{k})} + O(| \bm{k} h |^{r + 1}), \, \text{ as } h \rightarrow 0\, \text{ for all } \bm{k} \in \mathbb{Z}^d.
\end{equation*}
\end{theorem}
Here the $O(| \bm{k} h |^{r + 1})$ term is a matrix whose entries are $O(| \bm{k} h |^{r + 1})$ terms as $ h \rightarrow 0$. The ``only if" part is stated in theorem 2.1 of Lax's paper \cite{lax1961stability} for linear hyperbolic systems with variable coefficients. When the coefficients are constant, Lax's argument can also be used to show that the ``if" part is also true. 

We have the following theorem, which is an extension of theorem 4 in \cite{dupont2007back} to homogeneous linear hyperbolic systems with constant coefficients.

\begin{theorem}\label{thm:accuracy_system}
Suppose $Q_{\mathcal{L}^*}(\bm{k}) = \bar{Q}_{\mathcal{L}}(\bm{k})$ for any $\bm{k} \in \Z^d$ and scheme $\mathcal{L}$ is $r$-th order accurate for system \ref{eq:hyperbolic_system} with constant coefficient matrices, where $r$ is an odd integer, then the BFECC scheme $\mathcal{L}_{BFECC}$ based on $\mathcal{L}$ is $(r+1)$-th order accurate. 
\end{theorem}

\begin{proof}

Since $\mathcal{L}$ is $r$-th order accurate, by the Theorem-\ref{thm:lax} \cite{lax1961stability}, we have
\begin{equation*}
Q_{\mathcal{L}} = e^{\Delta t P(i \bm{k})} + Q_{r + 1}(i \bm{k} h) + O(| \bm{k} h |^{r + 2}),
\end{equation*}
where $Q_{r + 1}(i \bm{k} h)$ is a matrix with entries that are homogeneous degree $r+1$ polynomials in $i \bm{k}$ with real coefficients. 

By the assumption,
\begin{equation*}
Q_{\mathcal{L}^{*}} = \bar{Q}_{\mathcal{L}} = e^{- \Delta t P(i \bm{k})} + Q_{r + 1}(i \bm{k} h) + O(| \bm{k} h |^{r + 2}).
\end{equation*}

Then 
\begin{align*}
\bar{Q}_{\mathcal{L}} Q_{\mathcal{L}} = I + e^{- \Delta t P(i \bm{k})} Q_{r + 1}(i \bm{k} h)  + Q_{r + 1}(i \bm{k} h) e^{ \Delta t P(i \bm{k})} + O(| \bm{k} h |^{r + 2}).
\end{align*}

The Fourier symbol matrix $Q_B$ for $\mathcal{L}_{BFECC}$ is
\begin{align*}
Q_B & = Q_{\mathcal{L}} \left( I + \frac{1}{2} (I - \bar{Q}_{\mathcal{L}} Q_{\mathcal{L}}) \right)\\
  & = \left( e^{ \Delta t P(i \bm{k})} + Q_{r + 1}(i \bm{k} h) + O(| \bm{k} h |^{r + 2}) \right) \cdot \\
  & \left[ I - \frac{1}{2} \left(e^{- \Delta t P(i \bm{k})} Q_{r + 1}(i \bm{k} h)  + Q_{r + 1}(i \bm{k} h) e^{ \Delta t P(i \bm{k})}  \right) + O(| \bm{k} h |^{r + 2}) \right] \\
  & = e^{ \Delta t P(i \bm{k})} + \frac{1}{2} \left( Q_{r + 1}(i \bm{k} h)  - e^{ \Delta t P(i \bm{k})} Q_{r + 1}(i \bm{k} h) e^{ \Delta t P(i \bm{k})} \right) + O(| \bm{k} h |^{r + 2}) \\
  & = e^{ \Delta t P(i \bm{k})} + O(| \bm{k} h |^{r + 2}).
\end{align*}
Therefore $\mathcal{L}_{BFECC}$ is a $(r+1)$-th order accurate scheme.
\end{proof}

\subsection{Alternative view of BFECC method for hyperbolic PDE systems.}
In some cases, we can view the BFECC method for systems as applying the BFECC method for advection equations to the Riemann invariants. 

Consider a one dimensional hyperbolic PDE system with constant coefficients
\begin{align}\label{eq:u_eqn}
\partial_t \bm{u} + A \partial_x \bm{u} = 0.
\end{align}
For a hyperbolic system, the coefficient matrix $A$ is diagonalizable. Let 
$A = V \Lambda V^{-1}$, where $\Lambda$ is a diagonal matrix with eigenvalues of $A$ as entries, define $\bm{w} = V^{-1} \bm{u}$, then the system is equivalent to 
\begin{align}\label{eq:w_eqn}
\partial_t \bm{w} + \Lambda \partial_x \bm{w} = 0.
\end{align}

Suppose now we have a $r$-th order scheme $L$ for system (\ref{eq:w_eqn}), with $r$ being odd,
\begin{equation*}
\bm{W}^{n+1} = L \bm{W}^n.
\end{equation*}
Note that this scheme updates each component $W_i$ independently from other components. Then it gives a $r$-th order scheme $M$ for system-\ref{eq:u_eqn}, 
\begin{equation*}
\bm{U}^{n+1} = V \bm{W}^{n+1} = VL \bm{W}^n = VLV^{-1} \bm{U}.
\end{equation*}

By theorem 4 in \cite{dupont2007back}, applying BFECC to $L$ produces an $(r+1)$-th order scheme: 
\begin{equation*}
L_{B} = L \left( I + \frac{1}{2} ( I - \bar{L} L ) \right).
\end{equation*}

Applying BFECC to $M$ gives us
\begin{align*}
M_B  = M \left( I + \frac{1}{2} ( I - \bar{M} M ) \right) = V L_B V^{-1},
\end{align*}
therefore it is an $(r+1)$-th order scheme for system (\ref{eq:u_eqn}) following the results for scalar equations for $L_{B}$

However, not all schemes for system (\ref{eq:u_eqn}) come from schemes for system (\ref{eq:w_eqn}) that update components of $\bm{w}$ independently. Also, it is numerically more costly to decouple the system, especially in multi dimensions. In these cases, theorem \ref{thm:stability_system} and theorem \ref{thm:accuracy_system} provide the stability and accuracy improvement results. 

\section{BFECC schemes for the Maxwell's equations}\label{sec:bfecc_maxwell}

In this section, we discuss the BFECC schemes for the Maxwell's equations. We show that BFECC turns the central difference scheme and Lax-Friedrichs scheme into stable second order accurate schemes with larger CFL numbers than that of the Yee scheme on uniform rectangular grids. On non-orthogonal and unstructured grid, we discuss schemes based on least square linear approximation. 

Consider the dimensionless Maxwell's equations in a medium with zero conductivity \cite{taflove2005computational}
\begin{align}\label{eqn:maxwell_with_permittivity}
\begin{split}
& \epsilon_r \frac{\partial \bm{E}}{\partial t} = \nabla \times \bm{H} \\
& \mu_r \frac{\partial \bm{H}}{\partial t} = - \nabla \times \bm{E},
\end{split}
\end{align}
where $\epsilon_r$ and $\mu_r$ are the relative permittivity and permeability, respectively. We assume they are constant in the following discussion.

Let $\bm{E}'(t, \bm{x}) = \sqrt{\epsilon_r} \bm{E} (\sqrt{\epsilon_r \mu_r} t, \bm{x})$, $\bm{H}' (t, \bm{x}) = \sqrt{\mu_r} \bm{H}(\sqrt{\epsilon_r \mu_r} t, \bm{x})$, then the equations for $\bm{E}'$ and $\bm{H}'$ are
\begin{align*}
\begin{split}
&  \frac{\partial \bm{E}'}{\partial t} =  \nabla \times \bm{H}' \\
&  \frac{\partial \bm{H}'}{\partial t} = - \nabla \times \bm{E}'.
\end{split}
\end{align*}
To simplify the discussion for schemes, we use this Maxwell's equations in this section and refer to $\bm{E}'$ and $\bm{H}'$ as $\bm{E}$ and $\bm{H}$. 
\begin{align}\label{eqn:maxwell_without_permittivity}
\begin{split}
&  \frac{\partial \bm{E}}{\partial t} = \nabla \times \bm{H} \\
&  \frac{\partial \bm{H}}{\partial t} = -\nabla \times \bm{E}.
\end{split}
\end{align}

Note that in vacuum, we have $\epsilon_r = \mu_r = 1$, so system (\ref{eqn:maxwell_with_permittivity}) becomes (\ref{eqn:maxwell_without_permittivity}).

\subsection{BFECC based on the central difference scheme -- one dimensional case}
For simplicity, we consider Maxwell's equations in bounded domain $[0, 1]$ with periodic boundary conditions. The dimensionless Maxwell's equations in one dimensional free space are: 
\begin{align*}
\frac{\partial H_y}{ \partial t} = \frac{\partial E_z}{ \partial x} \\
\frac{\partial E_z}{ \partial t} = \frac{\partial H_y}{ \partial x}.
\end{align*}
For simplicity, denote $E = E_z, H = H_y$ and we have:
\begin{align*}
\frac{\partial H}{ \partial t} = \frac{\partial E}{ \partial x} \\
\frac{\partial E}{ \partial t} = \frac{\partial H}{ \partial x}. 
\end{align*}

The central difference scheme on a uniform rectangular grid for the above system is: 
\begin{align}\label{scheme:CD_1d}
\begin{split}
E^{n+1}_{j} = E^n_j + \frac{\lambda}{2} (H^n_{j+1} - H^n_{j-1}) \\ 
H^{n+1}_{j} = H^n_j + \frac{\lambda}{2} (E^n_{j+1} - E^n_{j-1}) 
\end{split}
\end{align}
where $\lambda = \Delta t / \Delta x$, $E_j^{n}$ and $H_{j}^{n}$ denote the numerical solutions $E_j^{n} \approx E(j \Delta x, t_n)$ and $H_j^{n} \approx H(j \Delta x, t_n)$.  

With periodic boundary conditions, $E^n_j$ and $H^n_j$ can be expanded uniquely as finite Fourier series: 
\begin{align*}
E^n_j = \sum_{k \in \mathcal{F}_N} C^n_k e^{2 \pi i k x_j} \\
H^n_j = \sum_{k \in \mathcal{F}_N} D^n_k e^{2 \pi i k x_j}
\end{align*}
where $k \in \mathcal{F}_{N}$ is the dual index, $C^n_k$ and $D^n_k$ are the Fourier coefficients for $E$ and $H$, respectively.  

Plug the finite Fourier series into the central difference scheme, we get
\begin{align*}
\begin{pmatrix}
C^{n+1}_k \\
D^{n+1}_k
\end{pmatrix}
 = Q_{\mathcal{L}} \begin{pmatrix}
C^{n}_k \\
D^{n}_k
\end{pmatrix}
= \begin{pmatrix}
1 & i \lambda \sin( 2 \pi k h ) \\
i \lambda \sin( 2 \pi k h ) & 1
\end{pmatrix}
\begin{pmatrix}
C^{n}_k \\
D^{n}_k
\end{pmatrix},
\end{align*}
where $Q_{\mathcal{L}}$ is the Fourier symbol matrix. Since the spectral radius of $Q_{\mathcal{L}}$ is greater than $1$ for most $\bm{k} \in \mathcal{F}_{\bm{N}}$, the central difference scheme is a first order scheme that is unstable and cannot be directly used to solve the Maxwell's equations. Applying BFECC method to the central difference scheme stabilizes it and also improves the order of accuracy to second order. 

Solving Maxwell's equations in the backward temporal direction is equivalent to changing $\lambda$ to $-\lambda$ in the scheme, therefore $Q_{\mathcal{L}^{*}}$ = $\overline{Q_{\mathcal{L}}}$. An easy calculation shows that $Q_{\mathcal{L}^{*}} Q_{\mathcal{L}} = Q_{\mathcal{L}}Q_{\mathcal{L}^{*}}$. The real and imaginary part of $Q_{\mathcal{L}}$ are both diagonalizable with real eigenvalues. Therefore the conditions of theorem \ref{thm:stability_system} and \ref{thm:accuracy_system} are satisfied. We see that BFECC based on the central difference scheme is $2$nd order accurate and $l^2$ stable if and only if $\rho(Q_{\mathcal{L}}) \leq 2$. Since the eigenvalues of $Q_{\mathcal{L}}$ are $1 \pm i \lambda \sin( 2 \pi k h)$, the stability condition reduces to $\max_{k \in \mathcal{F}_N} \left( 1 + \lambda^2 \sin^2( 2 \pi k h) \right) \leq 4 \Leftrightarrow \lambda \leq \sqrt{3}$. Therefore BFECC based on the central difference scheme is a 2nd order accurate scheme and is stable if $\Delta t / \Delta x \leq \sqrt{3}$.

An explicit calculation of the Fourier symbol matrix can be found in appendix \ref{app:bfecc_CD_details}, which verifies that it is $2$nd order accurate and stable if $\Delta t /\Delta x \leq \sqrt{3}$.

{\bf Remark}. In Section~\ref{sec:num_examples}, we apply schemes discussed in this section to Maxwell's equations with variable permittivities. The schemes discussed in this section can be simply adapted to the case with variable permittivities. For example, for the following system
\begin{align*}
\mu \frac{\partial H_y}{ \partial t} = \frac{\partial E_z}{ \partial x} \\
\epsilon \frac{\partial E_z}{ \partial t} = \frac{\partial H_y}{ \partial x},
\end{align*}
the central difference scheme is
\begin{align*}
\begin{split}
E^{n+1}_{i} = E^n_i + \frac{\lambda}{2 \mu_i} (H^n_{i+1} - H^n_{i-1}) \\ 
H^{n+1}_{i} = H^n_i + \frac{\lambda}{2 \epsilon_i} (E^n_{i+1} - E^n_{i-1}) 
\end{split}
\end{align*}
where $\epsilon_i$ and $\mu_i$ are the permittivity permeability respectively at grid point $x_i$. Other first order underlying schemes discussed in this paper can be similarly adapted to the variable coefficient case.

\subsection{BFECC based on the central difference scheme -- two dimensional case}\label{subsec:cd_2d}
Similar to the one dimensional case, we analyze BFECC based on the central difference scheme for the dimensionless Maxwell's equations in free space in the two dimensional $\rm{TM}_{\rm{z}}$ case. For simplicity, we consider the computational domain $[0, 1] \times [0, 1]$ with periodic boundary conditions. The Maxwell's equations are
\begin{align}\label{eqn:2d_maxwell_tmz}
\begin{split}
& \frac{\partial H_x}{\partial t} = - \frac{\partial E_z}{\partial y} \\
& \frac{\partial H_y}{\partial t} =  \frac{\partial E_z}{\partial x} \\
& \frac{\partial E_z}{\partial t} = \frac{\partial H_y}{\partial x} - \frac{\partial H_x}{\partial y}.
\end{split}
\end{align}

The central difference scheme is
\begin{align*}
& (H_x)^{n+1}_{i, j} = (H_x)^{n}_{i, j} - \frac{\lambda_y}{2} \left[ (E_z)^{n}_{i, j+1} - (E_z)^{n}_{i, j-1} \right] \\
& (H_y)^{n+1}_{i, j} = (H_y)^{n}_{i, j} + \frac{\lambda_x}{2} \left[ (E_z)^{n}_{i+1, j} - (E_z)^{n}_{i-1, j} \right] \\
& (E_z)^{n+1}_{i, j} = (E_z)^{n}_{i, j} + \frac{\lambda_x}{2} \left[ (H_y)^{n}_{i + 1, j} - (H_y)^{n}_{i-1, j} \right] -  \frac{\lambda_y}{2} \left[ (H_x)^{n}_{i, j+1} - (H_x)^{n}_{i, j-1} \right],
\end{align*}
where $\lambda_x = \Delta t / \Delta x$ and $\lambda_y = \Delta t / \Delta y$. 

Expand $H_x, H_y$ and $E_z$ into Fourier series: 
\begin{align*}
(H_x)^{n}_{j_1, j_2} = \sum_{(k, l) \in \mathcal{F}_{\bm N}} C^{n}_{k, l} e^{2 \pi i (k x_{j_1} + l y_{j_2})} \\
(H_y)^{n}_{j_1, j_2} = \sum_{(k, l) \in \mathcal{F}_{\bm N}} D^{n}_{k, l} e^{2 \pi i (k x_{j_1} + l y_{j_2})} \\
(E_z)^{n}_{j_1, j_2} = \sum_{(k, l) \in \mathcal{F}_{\bm N}} E^{n}_{k, l} e^{2 \pi i (k x_{j_1} + l y_{j_2})},
\end{align*}
where $(k, l) \in \mathcal{F}_{\bm N}$ are dual indices and $C^{n}_{k, l}$, $D^{n}_{k, l}$ and $E^{n}_{k, l}$ are Fourier coefficients for $H_x$, $H_y$ and $E_z$, respectively. 

Plug into the central difference scheme $\mathcal{L}$, we get
\begin{align*}
\begin{pmatrix}
C^{n+1}_{k, l} \\
D^{n+1}_{k, l} \\
E^{n+1}_{k, l}
\end{pmatrix} = & \, Q_{\mathcal{L}}
\begin{pmatrix}
C^{n}_{k, l} \\
D^{n}_{k, l} \\
E^{n}_{k, l}
\end{pmatrix},
\end{align*}
where 
\begin{align*}
Q_{\mathcal{L}} & = \begin{pmatrix}
1 & 0 & - i \lambda_y \sin( 2 \pi l \Delta y) \\
0 & 1 &  i \lambda_x \sin( 2 \pi k \Delta x) \\
- i \lambda_y \sin( 2 \pi l \Delta y) & i \lambda_x \sin( 2 \pi k \Delta x) & 1
\end{pmatrix} \\ 
& = I +  i \begin{pmatrix}
0 & 0 & -  \lambda_y \sin( 2 \pi l \Delta y) \\
0 & 0 &   \lambda_x \sin( 2 \pi k \Delta x) \\
-  \lambda_y \sin( 2 \pi l \Delta y) &  \lambda_x \sin( 2 \pi k \Delta x) & 0
\end{pmatrix} \\
& = I + i Y,
\end{align*}
and $Y = \rm{Im}(Q_{\mathcal{L}})$. Similar to the one dimensional case, solving the equation backward in time amounts to switching the signs of $\lambda_x$ and $\lambda_y$ in the scheme. Therefore we have $Q_{\mathcal{L}^{*}} = I - i Y = \overline{Q_{\mathcal{L}}}$, and $ Q_{\mathcal{L}^{*}} Q_{\mathcal{L}} = Q_{\mathcal{L}} Q_{\mathcal{L}^{*}} = I + Y^2$. $I$ and $Y$ are both symmetric real matrices, so they are diagonalizable with real eigenvalues. the conditions for theorem \ref{thm:stability_system} and \ref{thm:accuracy_system} are satisfied, and therefore BFECC based the central difference scheme is a $2$nd order accurate scheme and is stable if $\rho(Q_{\mathcal{L}}) \leq 2$. 

The eigenvalues of $Q_{\mathcal{L}}$ are
\begin{align*}
\lambda_1 = 1, \, \lambda_{2, 3} = 1 \pm i \sqrt{\lambda_x^2 (\sin(2 \pi k \Delta x) )^2 + \lambda_y^2 (\sin(2 \pi l \Delta y))^2}.
\end{align*}
The stability condition
\begin{align*}
& \rho(Q_{\mathcal{L}}) \leq 2, \, \forall (k, l) \in \mathcal{F}_{\bm{N}} \\
\Leftarrow \, & 1 + \lambda_x^2 (\sin(2 \pi k \Delta x) )^2 + \lambda_y^2 (\sin(2 \pi l \Delta y))^2 \leq 4, \, \forall (k, l) \in \mathcal{F}_{\bm{N}}. \\
\end{align*}
It is satisfied if
\begin{align}\label{eqn:cfl}
\lambda_x^2 + \lambda_y^2 \leq 3, \, \text{ or } \Delta t \leq \frac{\sqrt{3}}{\sqrt{(1/\Delta x)^2 + (1/\Delta y)^2}}.
\end{align}
If $\Delta x = \Delta y$, then $\Delta t \leq \frac{\sqrt{3}}{\sqrt{2}} \Delta x$ is sufficient for stability, which implies a CFL factor $\frac{\sqrt{3}}{\sqrt{2}} > 1$. 

An explicit calculation of the Fourier symbol matrix for the BFECC scheme is shown in appendix \ref{app:bfecc_CD_details}.

\subsection{BFECC based on the central difference scheme -- three dimensional case}
Similar to the one and two dimensional cases, we can also check the conditions of theorem \ref{thm:stability_system} and theorem \ref{thm:accuracy_system}, and find that BFECC based on the central difference scheme is second order accurate and $l^2$ stable if
\begin{align*}
\Delta t \leq \frac{\sqrt{3}}{\sqrt{(1/\Delta x)^2 + (1/\Delta y)^2 + (1/\Delta z)^2 }}.
\end{align*}
Note that this still implies a CFL factor equal to one in three dimensions if $\Delta x = \Delta y = \Delta z$. 

We summarize the results in the following theorem.
\begin{theorem}\label{thm:CD_thm}
BFECC based on the central difference scheme for Maxwell's equations in free space on uniform rectangular grid is second order accurate. It is stable in the $l^2$ sense if
\begin{enumerate}
\item in one dimensional case, $\Delta t \leq \sqrt{3} \Delta x $; or
\item in two dimensional case, $\Delta t \leq \frac{\sqrt{3}}{\sqrt{(1/\Delta x)^2 + (1/\Delta y)^2}}$; or
\item in three dimensional case, 
$
\Delta t \leq \frac{\sqrt{3}}{\sqrt{(1/\Delta x)^2 + (1/\Delta y)^2 + (1/\Delta z)^2 }}.
$
\end{enumerate} 
\end{theorem}

\subsection{BFECC based on the Lax-Friedrichs scheme}
We study BFECC based on the Lax-Friedrichs scheme $\mathcal{M}$ for the Maxwell's equations. In one dimension, the scheme is
\begin{align*}
E^{n+1}_{i} = \frac{E^n_{i-1} + E^n_{i+1}}{2} + \frac{\lambda}{2} (H^n_{i+1} - H^n_{i-1}) \\ 
H^{n+1}_{i} = \frac{H^n_{i-1} + H^n_{i+1}}{2} + \frac{\lambda}{2} (E^n_{i+1} - E^n_{i-1}). 
\end{align*} 

Write the one dimensional Maxwell's equations as 
\begin{align*}
\partial_t \bm{u} = A \partial_x \bm{u},
\end{align*}
where
\begin{align*}
\bm{u} = \begin{pmatrix}
E \\ H
\end{pmatrix} \, \text{ and } A = \begin{pmatrix}
0 & 1 \\
1 & 0
\end{pmatrix}.
\end{align*}
Then the Fourier symbol matrix of the Lax-Friedrichs scheme is $Q_{\mathcal{M}} = \cos(\tilde{k} h) I + i \lambda \sin(\tilde{k} h) A$, where $\tilde{k} = 2 \pi k$ is the angular wave number. It satisfies the conditions in theorem \ref{thm:stability_system} and theorem \ref{thm:accuracy_system}, so BFECC based on the Lax-Friedrichs scheme is second order accurate and is stable if and only if $| \rho(Q_{\mathcal{M}}) | \leq 2$, i.e.,
\begin{align*}
| \rho(Q_{\mathcal{M}}) |^2 = \cos^2( \tilde{k} h) + \lambda^2 \sin^2( \tilde{k} h) \leq 4. 
\end{align*}
This is true if $\lambda^2 \leq 4$. 

For two dimensional Maxwell's equations (\ref{eqn:2d_maxwell_tmz}), write the equations as
\begin{align*}
\partial_t \bm{u} = A_1 \partial_x \bm{u} + A_2 \partial_y \bm{u},
\end{align*}
where 
\begin{align*}
\bm{u} = \begin{pmatrix}
H_x \\ H_y \\ E_z
\end{pmatrix}, \, A_1 = \begin{pmatrix}
0 & 0 & 0 \\
0 & 0 & 1 \\
0 & 1 & 0
\end{pmatrix} \, \text{ and } A_2 = \begin{pmatrix}
0 & 0 & -1 \\
0 & 0 & 0 \\
-1 & 0 & 0
\end{pmatrix}.
\end{align*}

The Lax-Friedrichs scheme is
\begin{align*}
\bm{U}^{n+1}_{i, j}  = & \frac{\bm{U}^{n}_{i-1, j} + \bm{U}^{n}_{i+1, j} + \bm{U}^{n}_{i, j-1} + \bm{U}^{n}_{i, j+1} }{4}  \\
& + \frac{\Delta t}{2 \Delta x} A_1 \left(\bm{U}^{n}_{i+1, j} - \bm{U}^{n}_{i-1, j} \right) + \frac{\Delta t}{2 \Delta y} A_2 \left(\bm{U}^{n}_{i, j + 1} - \bm{U}^{n}_{i, j - 1} \right),
\end{align*}
where $\bm{U}^{n}_{i, j} \approx \left(
H_x(t_n, i \Delta x, j \Delta y), H_y(t_n, i \Delta x, j \Delta y), E_z(t_n, i \Delta x, j \Delta y) \right)^T$. 

Its Fourier symbol matrix is 
\begin{align*}
Q_{\mathcal{M}} = \frac{1}{2} \left( \cos(\tilde{k}_x h_x) + \cos(\tilde{k}_y h_y) \right) I + i \lambda_x \sin(\tilde{k}_x h_x ) A_1 + i \lambda_y \sin(\tilde{k}_y h_y ) A_2,
\end{align*}
where $\lambda_x = \Delta t / \Delta x$, $\lambda_y = \Delta t / \Delta y$, $h_x = \Delta x$, $h_y = \Delta y$ $\tilde{k}_x = 2 \pi k_x$, $\tilde{k}_y = 2 \pi k_y$, and $(k_x, k_y) \in \mathcal{F}_{\bm{N}}$.

\begin{align*}
| \rho(Q_{\mathcal{M}}) |^2  & = \frac{1}{4} \left( \cos(\tilde{k}_x h_x) +  \cos(\tilde{k}_y h_y) \right)^2 + \lambda_x^2 \sin^2(\tilde{k}_x h_x) + \lambda_y^2 \sin^2(\tilde{k}_y h_y) \\
& \leq \frac{1}{2} \left( \cos^2(\tilde{k}_x h_x) +  \cos^2(\tilde{k}_y h_y) \right) + \lambda_x^2 \sin^2(\tilde{k}_x h_x) + \lambda_y^2 \sin^2(\tilde{k}_y h_y) \\
& \leq \max \left( \frac{1}{2}, \lambda_x^2 \right) + \max\left( \frac{1}{2}, \lambda_y^2 \right) \\
& \leq \max \left( 1, \frac{1}{2} + \lambda_x^2,  \frac{1}{2} + \lambda_y^2, \lambda_x^2 + \lambda_y^2 \right) \leq 4.
\end{align*}
Therefore if
\begin{align*}
\max(\lambda_x, \lambda_y) \leq \sqrt{\frac{7}{2}}\, \text{ and } \lambda_x^2 + \lambda_y^2 \leq 4,
\end{align*}
then $| \rho(Q_{\mathcal{M}}) | \leq 2$ for any $(k_x, k_y) \in \mathcal{F}_{\bm{N}}$. The rest of the conditions of Theorem \ref{thm:stability_system} and \ref{thm:accuracy_system} can be easily verified.

Similarly, the stability condition for Maxwell's equations in three dimensions is
\begin{align*}
\max(\lambda_x, \lambda_y, \lambda_z) \leq \sqrt{3} \, \text{ and } \lambda_x^2 + \lambda_y^2 + \lambda_z^2 \leq 4,
\end{align*}
where $\lambda_x = \Delta t / \Delta x$, $\lambda_y = \Delta t / \Delta y$ and $\lambda_z = \Delta t / \Delta z$.
 
The stability and accuracy results are summerized as follows: 
\begin{theorem}\label{thm:LF_thm}
BFECC based on the Lax-Friedrichs scheme for Maxwell's equations in free space on uniform rectangular grid is $2$nd order accurate. It is stable in the $l^2$ sense if
\begin{enumerate}
\item in one-dimensional case, $\Delta t \leq 2 \Delta x$; or
\item in two-dimensional case, $\Delta t \leq \frac{2}{\sqrt{(1/\Delta x)^2 + (1/\Delta y)^2}}$ and $\Delta t \leq \sqrt{ \frac{7}{2} } \min( \Delta x, \Delta y)$; or 
\item in three-dimensional case, 
\begin{align*}
\Delta t \leq \frac{2}{\sqrt{(1/\Delta x)^2 + (1/\Delta y)^2 + (1/\Delta z)^2 }} \, \text{ and } \Delta t \leq \sqrt{3} \min(\Delta x, \Delta y, \Delta z).
\end{align*}
\end{enumerate} 
\end{theorem} 

\subsection{BFECC based on interpolation of the central difference and the Lax-Friedrichs schemes}
The Lax-Friedrichs schems is more diffusive than the central difference scheme as the underlying scheme for BFECC. However, when there are discontinuities in the coefficients of the equations, the latter scheme may generate some numerical artifacts in the vicinities of the discontinuities. An interpolation between the two schemes could combine the strengths of both schemes. Let $\theta \in [0, 1]$. A $\theta$-scheme $\mathcal{L}_{\theta}$ is formally $\mathcal{L}_{\theta} = (1 - \theta) \mathcal{L} + \theta \mathcal{M}$, where $\mathcal{L}$ is the central difference scheme and $\mathcal{M}$ is the Lax-Friedrichs scheme for Maxwell's equations. 

Using aforementioned notations, for one dimensional Maxwell's equations, the scheme is
\begin{align*}
\bm{U}^{n+1}_{i} = (1 - \theta) \bm{U}^{n}_{i} + \theta \frac{\bm{U}^{n}_{i-1} + \bm{U}^{n}_{i+1} }{2} +  \frac{\Delta t}{2 \Delta x} A \left(\bm{U}^{n}_{i+1} - \bm{U}^{n}_{i-1} \right),
\end{align*}
where 
\begin{align*}
{U}^{n}_{i} \approx \begin{pmatrix}
E(t_n, i \Delta x) \\
H(t_n, i \Delta x)
\end{pmatrix}\, \text{ and } A = \begin{pmatrix}
0 & 1 \\
1 & 0
\end{pmatrix}.
\end{align*}
Its Fourier symbol matrix is 
\begin{align*}
Q_{\theta} = \left( 1 - \theta + \theta \cos(\tilde{k} h) \right) I + i \lambda \sin(\tilde{k} h) A,
\end{align*}
which satisfies all the conditions in Theorem \ref{thm:stability_system} and \ref{thm:accuracy_system}, and BFECC based on the $\theta$-scheme is second order accurate. Note that
\begin{align*}
| \rho(Q_{\theta}) |^2 = \left[ (1 - \theta) + \theta \cos(\tilde{k} h) \right]^2 + \lambda^2 \sin^2 (\tilde{k} h ).
\end{align*}
Since $f(x) = x^2$ is convex, we have
\begin{align*}
\left[ (1 - \theta) + \theta \cos(\tilde{k} h) \right]^2 \leq (1 - \theta) + \theta \cos^2(\tilde{k} h).
\end{align*}
Therefore
\begin{align*}
| \rho(Q_{\theta}) |^2  \leq (1 - \theta) + \theta \cos^2(\tilde{k} h) + \lambda^2 \sin^2 (\tilde{k} h) = (1 - \theta) | \rho(Q_{\mathcal{L}}) |^2 + \theta | \rho(Q_{\mathcal{M}}) |^2,
\end{align*}
where $Q_{\mathcal{L}}$ and $Q_{\mathcal{M}}$ are the Fourier symbol matrices for the central difference and Lax-Friedrichs schemes, respectively. And the CFL number of the $\theta$-scheme is between $\sqrt{3}$ and $2$. 

Similarly, for the two dimensional Maxwell's equations (\ref{eqn:2d_maxwell_tmz}), the $\theta$-scheme is
\begin{align*}
\bm{U}^{n+1}_{i, j}  = & (1 - \theta) \bm{U}^{n}_{i, j} + \theta \frac{\bm{U}^{n}_{i-1, j} + \bm{U}^{n}_{i+1, j} + \bm{U}^{n}_{i, j-1} + \bm{U}^{n}_{i, j+1} }{4}  \\
& + \frac{\Delta t}{2 \Delta x} A_1 \left(\bm{U}^{n}_{i+1, j} - \bm{U}^{n}_{i-1, j} \right) + \frac{\Delta t}{2 \Delta y} A_2 \left(\bm{U}^{n}_{i, j + 1} - \bm{U}^{n}_{i, j - 1} \right),
\end{align*}
where $\bm{U}^{n}_{i, j}$, $A_1$ and $A_2$ are defined as in Section~\ref{subsec:cd_2d}. 
And its Fourier symbol matrix 
\begin{align*}
Q_{\theta} = q_{\theta} I + i \lambda_x \sin( \tilde{k}_x h_x ) A_1 + i \lambda_y \sin( \tilde{k}_y h_y ) A_2,
\end{align*}
where $q_{\theta} = \left[ 1 - \theta + \theta \frac{\cos( \tilde{k}_x h_x ) + \cos(\tilde{k}_y h_y )  }{2} \right]$. The spectral radius $\rho(Q_{\theta})$ satisfies
\begin{align*}
| \rho(Q_{\theta}) |^2  = q_{\theta}^2 + \lambda_x^2 \sin^2( \tilde{k}_x h_x ) + \lambda_y^2 \sin^2( \tilde{k}_y h_y ) \leq (1 - \theta) | \rho(Q_{\mathcal{L}}) |^2 + \theta | \rho(Q_{\mathcal{M}}) |^2.
\end{align*}
In the inequality, we again use the convexity of $f(x) = x^2$ and the special form of $q_{\theta}$. Therefore, the constant in the CFL condition similar to (\ref{eqn:cfl}) would be between $\sqrt{3}$ and $2$. The analysis for three dimensional Maxwell's equations is similar, and the result is summarized as follows. 
\begin{theorem}\label{thm:interpolated_scheme}
Let $\theta \in [0, 1]$, and $\mathcal{L}_{\theta} = (1 - \theta) \mathcal{L} + \theta \mathcal{M}$, where $\mathcal{L}$ is the central difference scheme and $\mathcal{M}$ is the Lax-Friedrichs scheme for Maxwell's equations. Then BFECC based on $\mathcal{L}_{\theta}$ is second order accurate. It is stable if

\begin{enumerate}
\item in one dimensional case, $\Delta t \leq c_{\theta} \Delta x$; or
\item in two dimensional case, $\Delta t \leq \frac{c_{\theta}}{\sqrt{(1/\Delta x)^2 + (1/\Delta y)^2}}$ and $\Delta t \leq \sqrt{ \frac{7}{2} } \min( \Delta x, \Delta y)$; or 
\item in three dimensional case, 
\begin{align*}
\Delta t \leq \frac{c_{\theta}}{\sqrt{(1/\Delta x)^2 + (1/\Delta y)^2 + (1/\Delta z)^2 }}\, \text{ and } \Delta t \leq \sqrt{3} \min(\Delta x, \Delta y, \Delta z),
\end{align*}
\end{enumerate}
where $c_{\theta} \in [\sqrt{3}, 2]$ depends only on $\theta$. 
\end{theorem}

\subsection{Least square local linear approximation for non-rectangular grids}\label{subsec:ls_scheme}

A special case of $\mathcal{L}_{\theta}$ is based on the linear least square fitting, which can also be used on irregular grids conveniently. In order to adapt to non-orthogonal grids, we consider some simple first order underlying schemes based on linear least squares. Least squares method significantly improves the robustness of polynomial approximation in multi dimensions. In WENO-type schemes for solving nonlinear conservation laws on unstructured meshes, least squares (high degree) polynomial fitting has been used, see for example \cite{barth1990higher, hu1999weighted}. 

To design an explicit scheme for the Maxwell's equations, we need approximations for spatial derivatives such as $\frac{\partial E_z }{\partial x}$ and $\frac{\partial H_x}{\partial y}$ at the time $t_n$ to update field variables $\bm{E}$ and $\bm{H}$. A natural approach is to locally fit a linear function for each component of a field variable using the function values at a grid point and its neighbors, and then use the spatial derivatives of the linear function as approximations. 

Consider for example the approximation of $H_x$ and its derivatives at a grid point $(x_i, y_j)$. Denote this point $(x^0, y^0)$. Suppose its neighboring grid points are $(x^1, y^1)$, $(x^2, y^2)$, ..., $(x^K, y^K)$, where $K \geq 2$, and denote $(H_x)^i = H_x (x^i, y^i)$ for $i = 0, 1, ..., K$. A linear function $\hat{H}_x(x, y) = \hat{a} + \hat{b} (x - x^{0}) + \hat{c} ( y - y^{0})$ can be determined to fit the numerical values of $H_x$ at $(x^{j}, y^{j})$, $j = 0, 1, ..., K$, by using least squares fitting. This is a local procedure, and has to be done at every point at which the scheme is evaluated. 

We denote the approximated spatial derivatives at $(x^0, y^0)$ by $\frac{\partial \hat{H}_x}{\partial x}$ and  $\frac{\partial \hat{H}_x}{\partial y}$, and the approximated function value at $(x^0, y^0)$ by $\hat{H}_x(x^0, y^0)$ or $\left( \hat{H}_x \right)_{i, j}$. 

Similarly let $\frac{\partial \hat{E}_z}{\partial x}$, $\frac{\partial \hat{E}_z}{\partial y}$, $\frac{\partial \hat{H}_y}{\partial x}$, $\frac{\partial \hat{H}_y}{\partial y}$ be the least square approximation of $E_z$ and $H_y$'s partial derivatives at $(x^{0}, y^{0})$. An explicit scheme similar to the central difference scheme is (for Maxwell's equations in two dimensions, (\ref{eqn:2d_maxwell_tmz})):
\begin{align}\label{scheme:ls_cd}
\begin{split}
\left( E_z \right)^{n+1}_{i, j} = & \left( E_z \right)^{n}_{i, j} + \Delta t \left(  \left( \frac{\partial \hat{H}_y}{\partial x} \right)_{i, j}^n - \left( \frac{\partial \hat{H}_x}{\partial y} \right)_{i, j}^n \right)  \\
\left( H_x \right)^{n+1}_{i, j} = & \left( H_x \right)^{n}_{i, j} - \Delta t \left( \frac{\partial \hat{E}_z}{\partial y} \right)_{i, j}^n  \\
\left( H_y \right)^{n+1}_{i, j} = & \left( H_y \right)^{n}_{i, j} + \Delta t \left( \frac{\partial \hat{E}_z}{\partial x} \right)_{i, j}^n,
\end{split}
\end{align}
where $\left(E_z \right)_{i, j}^{n}$ denotes the numerical solution $\left(E_z \right)_{i, j}^{n} \approx E_z(x_i, y_j, t_n)$, similarly for $\left(H_x \right)_{i, j}^{n}$ and $\left(H_y \right)_{i, j}^{n}$. The set of grid points near $(x_i, y_j)$ used for least squares fitting in this paper are $(x_i, y_j)$, $(x_{i \pm 1}, y_j)$ and $(x_i, y_{j \pm 1})$. When the grid is a uniform rectangular grid, then the above least square approximation for spatial derivatives is the central difference approximation if the same set of neighboring points are used, and (\ref{scheme:ls_cd}) is just the central difference scheme. We refer to (\ref{scheme:ls_cd}) as the least square central difference scheme.

One could use the least square approximated field values as well as the least square approximated derivatives in the scheme, i.e. 
\begin{align}\label{scheme:ls_theta}
\begin{split}
\left( E_z \right)^{n+1}_{i, j} = & \left( \hat{E}_z \right)^{n}_{i, j} + \Delta t \left(  \left( \frac{\partial \hat{H}_y}{\partial x} \right)_{i, j}^n - \left( \frac{\partial \hat{H}_x}{\partial y} \right)_{i, j}^n \right)  \\
\left( H_x \right)^{n+1}_{i, j} = & \left( \hat{H}_x \right)^{n}_{i, j} - \Delta t \left( \frac{\partial \hat{E}_z}{\partial y} \right)_{i, j}^n  \\
\left( H_y \right)^{n+1}_{i, j} = & \left( \hat{H}_y\right)^{n}_{i, j} + \Delta t \left( \frac{\partial \hat{E}_z}{\partial x} \right)_{i, j}^n .
\end{split}
\end{align}
Here the subscript $(i, j)$ and superscript $n$ indicate that the approximation is done in a neighborhood of grid point $(x_i, y_j)$ using field values at time level $t_n$. Note $\left( \hat{E}_z \right)^{n}_{i, j}$, $\left( \hat{H}_x \right)^{n}_{i, j}$ and $\left( \hat{H}_y \right)^{n}_{i, j}$ are weighted averages of field values at $(i, j)$ and its neighbors, therefore this scheme is similar to the $\theta$-scheme on uniform rectangular grids. When the grid is a uniform rectangular grid (possibly with $\Delta x \neq \Delta y$), scheme (\ref{scheme:ls_theta}) reduces to the $\theta$-scheme with $\theta = 0.8$. We refer to this scheme as the least square $\theta$-scheme.

Both schemes are first order accurate, because the least square gradient approximation are first order accurate, and the least square field value approximation is second order accurate. Function approximation by least squares fitting have been well studied (see e.g. \cite{burden2001numerical, trefethen2013approximation}). For completeness, we give a short discussion on the accuracy of the least squares fitting. Without loss of generality, we can assume $(x^0, y^0) = (0, 0)$.  In a neighborhood of $(0, 0)$ with radius $O(h)$, rewrite function $u(x, y)$ as 
\begin{align*}
u(x, y) = a + b x + c y + f(x, y) = l(x, y) + f(x, y),
\end{align*}
where $f(x, y) = O(x^2 + y^2)$.
Suppose the linear function to be determined by least squares is
\begin{align*}
\hat{u}(x, y) = \hat{a} + \hat{b} x + \hat{c} y.
\end{align*}
We would like to show $|| (\hat{a}, \hat{b}, \hat{c}) - (a, b, c) || = O(h)$. Denote $\theta = (a, b, c)^T$ and $\hat{\theta} = (\hat{a}, \hat{b}, \hat{c})^T$. Suppose $(x^0, y^0)$'s neighboring grid points are $(x^1, y^1), (x^2, y^2), ..., (x^K, y^K)$, satisfying $\sqrt{ \left(x^j\right)^2 + \left(y^j\right)^2} = O(h)$, for $j = 1, 2, ..., K$. The coordinates of these points are collected in matrix $A$,
\begin{align}\label{eqn:grid_point_matrix}
A = \begin{pmatrix}
1 & x^0 & y^0 \\
1 & x^1 & y^1 \\
... & ... & ... \\
1 & x^K & y^K
\end{pmatrix},
\end{align}
and function values at these grid points are collected in vector $U = L + F$, where $L = \left( l(x^0, y^0), ..., l(x^K, y^K) \right)^T$ and $F = \left( f(x^0, y^0), ..., f(x^K, y^K) \right)^T$.
Then we have
\begin{align*}
\hat{\theta} = (A^T A)^{-1}A^T U \\
\theta = (A^T A)^{-1}A^T L.
\end{align*}
Therefore 
\begin{align*}
A(\hat{\theta} - \theta) = A (A^T A)^{-1}A^T ( U - L) = A (A^T A)^{-1}A^T F \\
\Rightarrow || A(\hat{\theta} - \theta) || = || A (A^T A)^{-1}A^T F || \leq || F ||, 
\end{align*}
where $|| \cdot ||$ denotes the $l^2$ norm. In the above, we use the fact that $A (A^T A)^{-1}A^T$ is an orthogonal projection.

Suppose $A$ is a $(K+1) \times 3$ matrix of full rank, so its smallest singular value $\sigma_3(A) > 0$. Suppose $\sigma_3(A) \geq D h$ for some constant $D > 0$, then we have 
\begin{align*}
& D h || (\hat{\theta} - \theta) || \leq \sigma_3 (A) || (\hat{\theta} - \theta) || \leq ||A(\hat{\theta} - \theta)|| \leq || F || \leq C \sqrt{K+1} h^2 \\
\Rightarrow &  || (\hat{\theta} - \theta) || \leq \frac{C \sqrt{K+1}}{D} h.
\end{align*}

So the problem reduces to a geometric condition $\sigma_3(A) \geq D h$ for some $D > 0$ for the selected neighboring grid points. It can be easily verified that the rectangular mesh and the hexagonal mesh both satisfy this condition. For example, a rectangular grid of size $h$ has $\sigma_3(A) = 2 h$, and a uniform hexagonal grid with edge length $h$ has $\sigma_3(A) = \sqrt{3} h$ (using a grid point and its $6$ adjacent grid points in the least squares fitting). 

Next, to show the least square field value approximation is second order accurate, we notice that $(x^0, y^0) = (0, 0$ and the first component of $A(\hat{\theta} - \theta)$ is 
\begin{align*}
\hat{a} + \hat{b} x^0 + \hat{c} y^0 - (a + b x^0 + c y^0) = \hat{a} - a = \hat{u}^{0} - u(x^0, y^0),
\end{align*}
where $\hat{u}^{0}$ is the least square field value approximation. Therefore
\begin{align*}
| \hat{u}^{0} - u(x^0, y^0)| \leq ||A(\hat{\theta} - \theta)|| \leq C \sqrt{K+1} h^2.
\end{align*}
Therefore the least square field value approximation is second order accurate. 

The order of accuracy result is summarized in the following theorem.
\begin{theorem}
Suppose the grid points coordinate matrix $A$ defined in (\ref{eqn:grid_point_matrix}) satisfies $\sigma_3(A) \geq D h$ for some positive constant $D$, then the least square center difference scheme and the least square $\theta$-scheme are both first order accurate.
\end{theorem}

Similar to the central difference scheme, the least square central difference scheme is usually numerically unstable. We can apply the BFECC method to improve the stability and accuracy. The least square $\theta$-scheme is conditionally stable, and applying BFECC also improves its stability and accuracy. On a uniform rectangular grid, BFECC based on the least square central difference and least square $\theta$-scheme are second order accurate and stable with CFL number $\sqrt{3}$ and a CFL number between $\sqrt{3}$ and $2$, respectively. On non-uniform or non-orthogonal grids, our current analysis is not sufficient to prove the stability and order of accuracy. Numerical examples in Section \ref{sec:num_examples} show that BFECC based on the least square $\theta$-scheme is conditionally stable and second order accurate. We omit the examples for BFECC based on the least square central difference scheme, which is also second order in our experiments with smooth solutions (not reported here) but is likely to have numerical artifacts at places where the coefficients of the equations have jump discontinuities.

{\bf Remark}.
As will be discussed in Section~\ref{subsec:div_free}, on a uniform rectangular grid, the central difference scheme and the BFECC scheme based on it preserve the divergence free property of the magnetic field. On a non-rectangular grid, the least square schemes and the corresponding BFECC schemes don't have this property. The flexibility of least square gradient approximation allows an option to reduce the divergence error. We can add a penalty term $\lambda \left(\left( \frac{\partial \hat{H}_x}{\partial x} \right) + \left( \frac{\partial \hat{H}_y}{\partial y} \right)\right)^2$ to the minimization functional of the least squares method, where $\lambda \geq 0$ is a parameter. The Gauss's law for the electric field can similarly be incorporated into the least squares. We will study them in the future. 

\subsection{Point shifted algorithm for grid generation}
It is often necessary to model curved material interfaces in computational eletromagnetics. The simplest treatment with a staircased approximation for the curved boundary can lead to large errors \cite{cangellaris1991analysis, taflove2005computational}. Local subcell methods \cite{taflove2005computational} model curved interfaces/boundaries by modifying the update rule near them. In these cells,  the integral form of the Maxwell's equations are usually used to update the field, e.g., the contour path method \cite{jurgens1992finite}. 

Using BFECC based on the least square central difference scheme (\ref{scheme:ls_cd}) or BFECC based on the least square $\theta$-scheme (\ref{scheme:ls_theta}), we can locally deform the grid near a curved interface to conform with the interface, and avoid switching to the integral form of the Maxwell's equations in these deformed cells. In this section, we describe a simple point shifted algorithm \cite{mcbryan1980elliptic} for shifting nearby grid points to the interface. It is used for numerical examples of scattering in Section~\ref{sec:num_examples}.

Given a uniform rectangular grid in two dimensions, denote the grid points $G_{\rm{rec}} = \{ (x_i, y_j): x_i = i \Delta x, y_j = j \Delta y, i = 0, 1, ..., N_{x}, j = 0, 1, ..., N_{y} \}$. Let $C$ be a closed curve, e.g., the boundary of a scattering object. The point shifted algorithm shifts nearby grid points to the interface for distances less than half of the grid size so that the topological structure of the grid remains unchanged. And the new grid point set $G_C = \{ (\tilde{x}_i, \tilde{y}_j): i = 0, 1, ..., N_x, j = 0, 1, ..., N_{y} \}$ conforms with curve $C$. It does so by finding the intersections of the grid lines and $C$, and shifts the nearest grid points to the intersection points.

\begin{algorithm}\label{alg:point_shift}
    \SetKwInOut{Input}{Input}
    \SetKwInOut{Output}{Output}

    \Input{Rectangular grid $G_{\rm{rec}} = \{ (x_i, y_j): x_i = i \Delta x, y_j = j \Delta y, i = 0, 1, ..., N_{x}, j = 0, 1, ..., N_{y} \}$, and a curve $C$.}
    \Output{Deformed grid $G_C = \{ (\tilde{x}_i, \tilde{y}_j): i = 0, 1, ..., N_x, j = 0, 1, ..., N_{y} \}$.}
    1. Copy $G_{\rm{rec}}$ to $G_C$: set $\tilde{x}_i = x_i$, $\tilde{y}_j = x_j$ for $i = 0, 1, ..., N_x$, $j = 0, 1, ..., N_y$\;
    2. Find all intersection points $\{ (\hat{x}^k, \hat{y}^k): k = 0, 1, ..., K\}$ on grid lines cut by $C$\;
    3. \For{k = 0, 1, ..., K}{
    Find the nearest point $(x_{i^*}, y_{j^{*}})$ in $G_{\rm{rec}}$ to $(\hat{x}^k, \hat{y}^k)$, when there is a tie, break the tie arbitrarily. Set $(\tilde{x}_{i^*}, \tilde{y}_{j^{*}}) = (\hat{x}^k, \hat{y}^k)$. 
    }
    4. Return $G_C$. 
    \caption{Point shifted algorithm}
\end{algorithm}

{\bf Remark}. A optional smoothing step can be added after the point shift to make the grid deformation more smooth. Denote the uniform rectangular grid points $\bm{x}_{i, j}$ and the point shifted grid point $\tilde{\bm{x}}_{i, j}$, where $i = 0, 1, ..., N_x$ and $j = 0, 1, ..., N_y$. First compute the point shift deformation $\bm{d}_{i, j} = \tilde{\bm{x}}_{i, j} - \bm{x}_{i, j}$. Second, copy $\bm{d}_{i, j}$ to $\tilde{\bm{d}}_{i, j}$, and 
for every $(i, j)$ such that $\bm{d}_{i, j} = \bm{0}$ (i.e. unshifted points), set 
\begin{align*}
\tilde{\bm{d}}_{i, j} = \frac{\bm{d}_{i-1, j} + \bm{d}_{i+1, j} + \bm{d}_{i, j-1} + \bm{d}_{i, j+1}}{4}. 
\end{align*}
This has the effect of smoothing out the point shift deformation. Third, assign new locations to the shifted grid points
\begin{align*}
\tilde{\bm{x}}_{i, j} = \bm{x}_{i, j} + \tilde{\bm{d}}_{i, j}.
\end{align*}
for $i = 0, 1, ..., N_x$ and $j = 0, 1, ..., N_y$. Note the shifted grid points that lie on the curve $C$ are unaffected by this smoothing step, only their neighbors get shifted in the smoothing step. This step can be repeated multiple times to smooth out the deformation to points that are further away from the curve $C$. Smoothing helps reduce grid deformation near the interface, and can be helpful when complicated interfaces are involved.

Figure \ref{fig:point_shifted_grids} shows examples of non-rectangular grids after applying the point shifted algorithm. The subfigure (a) is a uniform rectangular grid shifted to conform a circle without smoothing, the subfigure (b) is the same grid shifted to conform a circle, with a smoothing step, and the subfigure (c) is a uniform rectangular grid shifted to conform a more complicated curve, without smoothing. Grid (a) and (c) are use in the scattering numerical examples in Section~\ref{sec:num_examples}. We didn't use the smoothing step since the material interfaces in our numerical examples are simple and solutions on grids without smoothing already has expected order of accuracy. Note that the topologies of these grids have not been changed by the algorithm, making the implementation almost as simple as on a uniform rectangular grid. 

\begin{figure}[H]
\includegraphics[width=0.7\linewidth]{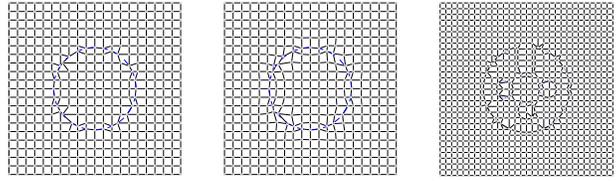}
\caption{Point shifted grids.}
\label{fig:point_shifted_grids}
\end{figure}

\subsection{Divergence of the magnetic field}\label{subsec:div_free}
The magnetic field satisfies the divergence free condition in the Maxwell's equations as long as it does so initially. We show that the central difference scheme conserves the numerical divergence of the magnetic field when the grid is a uniform rectangular grid. Therefore, BFECC based on the central difference scheme also conserves the numerical divergence of the magnetic field.

The numerical divergence of the magnetic field at the time $t_n$ is:
\begin{align*}
(\nabla \cdot \vec{H})^n_{i, j} = \frac{(H_x)^n_{i + 1, j} - (H_x)^n_{i - 1, j}}{2 \Delta x} + \frac{(H_y)^n_{i, j + 1} - (H_y)^n_{i, j - 1}}{2 \Delta y}.
\end{align*}

Using the central difference scheme to update $H_x$ and $H_y$, we get:
\begin{align*}
 \frac{(H_x)^{n+1}_{i + 1, j} - (H_x)^{n+1}_{i - 1, j}}{2 \Delta x} = & \frac{(H_x)^n_{i + 1, j} - (H_x)^n_{i - 1, j}}{2 \Delta x} - \\
 & \frac{(E_z)^n_{i+1, j+1} -(E_z)^n_{i+1, j-1} - (E_z)^n_{i-1, j+1} + (E_z)^n_{i-1, j-1} }{4 \Delta x \Delta y} \Delta t
\end{align*}
and 
\begin{align*}
\frac{(H_y)^{n+1}_{i, j + 1} - (H_y)^{n+1}_{i, j - 1}}{2 \Delta y} = & \frac{(H_y)^n_{i, j + 1} - (H_y)^n_{i, j - 1}}{2 \Delta y} + \\
  & \frac{(E_z)^n_{i+1, j+1} -(E_z)^n_{i+1, j-1} - (E_z)^n_{i-1, j+1} + (E_z)^n_{i-1, j-1} }{4 \Delta x \Delta y} \Delta t.
\end{align*}

Therefore
\begin{align*}
(\nabla \cdot \bm{H})^{n + 1}_{i, j} = (\nabla \cdot \bm{H})^n_{i, j}.
\end{align*}

Similar arguments show that for the Lax-Friedrichs scheme and the $\theta$-scheme, $\left( \nabla \cdot \bm{H} \right)_{i, j}^{n+1}$ is a convex combination of $\nabla \cdot \bm{H}$ at $(x_i, y_j)$ and its neighboring grid points (these two schemes only conserve $\sum_{i, j} \left( \nabla \cdot \bm{H} \right)_{i, j}$). In particular, if $\left( \nabla \cdot \bm{H} \right)_{i, j} = 0$ for all $i$ and $j$ initially, then this property holds for all subsequent $t_n$.

For irregular grids, the divergence free property is no longer guaranteed. But divergence penalty terms can be added to the minimization functional in the least square gradient approximation to reduce the divergence error, as discussed in Section~\ref{subsec:ls_scheme}.

\subsection{Perfectly Matched Layer}
Perfectly matched layers are commonly used as the absorbing boundary condition for problems in unbounded domains \cite{berenger1994perfectly}. We consider combining the unsplit convolutional perfectly matched layer \cite{komatitsch2007unsplit} with the BFECC method. Here we adapt the implementation in \cite{schneider2010understanding} and discuss it in the two dimensional case. The three dimensional case will be similar.

In the lossless domain, consider
\begin{align*}
& \frac{\partial E_z}{\partial t} = \frac{\partial H_y}{\partial x} - \frac{\partial H_x}{\partial y} \\
& \frac{\partial H_x}{\partial t} = - \frac{\partial E_z}{\partial y} \\
& \frac{\partial H_y}{\partial t} =  \frac{\partial E_z}{\partial x}.
\end{align*}

With the unsplit convolutional perfectly match layers, the equations in the perfectly matched layers are:
\begin{align*}
& \frac{\partial E_z}{\partial t} = \frac{\partial H_y}{\partial x} - \frac{\partial H_x}{\partial y}  + \zeta_x(t) \ast \frac{\partial H_y}{\partial x}  - \zeta_y(t) \ast \frac{\partial H_x}{\partial y}\\
& \frac{\partial H_x}{\partial t} = - \frac{\partial E_z}{\partial y} - \zeta_y (t) \ast \frac{\partial E_z}{\partial y} \\
& \frac{\partial H_y}{\partial t} =  \frac{\partial E_z}{\partial x} + \zeta_x(t) \ast \frac{\partial E_z}{\partial x},
\end{align*}
where 
\begin{equation*}
\zeta_w (t) = - \sigma_w e^{- \sigma_w t} u(t), \, w = x, y,
\end{equation*}
$u(t)$ is the unit step function, and $\sigma_x, \sigma_y$ are chosen conductivity parameters in the perfectly matched layers (PMLs). For PMLs adjacent to a boundary perpendicular to the $x$-axis, we choose $\sigma_x > 0$ and $\sigma_y = 0$; and for PMLs adjacent to a boundary perpendicular to the $y$-axis, we choose $\sigma_y > 0$ and $\sigma_x = 0$.

To implement BFECC in the perfectly matched layers, we first denote 
\begin{align*}
& b_x = e^{-\sigma_x \Delta t}, \, b_y =  e^{-\sigma_y \Delta t} \\
& c_x = b_x - 1, \, c_y = b_y - 1 \\
& \left(\Psi_{E_z x}\right)_{i, j}^n = \left( \zeta_x(t) \ast \frac{\partial H_y}{\partial x} \right)^{n}_{i, j} \\
& \left(\Psi_{E_z y}\right)_{i, j}^n = \left( \zeta_y(t) \ast \frac{\partial H_x}{\partial y} \right)^{n}_{i, j} \\
& \left(\Psi_{H_x y}\right)_{i, j}^n = \left( \zeta_y(t) \ast \frac{\partial E_z}{\partial y} \right)^{n}_{i, j} \\
& \left(\Psi_{H_y x}\right)_{i, j}^n = \left( \zeta_y(t) \ast \frac{\partial E_z}{\partial x} \right)^{n}_{i, j}.
\end{align*}

To update the field variables in the PMLs, we separate the terms dependent on the time $t_n$ from others in the convolution integrals (in order to stabilize them later by BFECC), approximate them using least squares, and obtain the least square central difference scheme:
\begin{align*}
\left( E_z \right)^{n+1}_{i, j} = & \left( E_z \right)^{n}_{i, j} + \Delta t \left( \left( \frac{\partial \hat{H}_y}{\partial x} \right)_{i, j}^n - \left( \frac{\partial \hat{H}_x}{\partial y} \right)_{i, j}^n \right) \\ + & \left( c_x \left( \frac{\partial \hat{H}_y}{\partial x} \right)_{i, j}^n + b_x \left(\Psi_{E_z x}\right)_{i, j}^{n-1} \right) \Delta t \\ -&  \left( c_y \left( \frac{\partial \hat{H}_x}{\partial y} \right)_{i, j}^n + b_y \left(\Psi_{E_z y}\right)_{i, j}^{n-1} \right) \Delta t
\end{align*}
\begin{align*}
\left( H_x \right)^{n+1}_{i, j} = & \left( H_x \right)^{n}_{i, j} - \Delta t \left( \frac{\partial \hat{E}_z}{\partial y} \right)_{i, j}^n  \\ - & \left( c_y \left( \frac{\partial \hat{E}_z}{\partial y} \right)_{i, j}^n + b_y \left(\Psi_{H_x y}\right)_{i, j}^{n-1} \right) \Delta t
\end{align*}
\begin{align*}
\left( H_y \right)^{n+1}_{i, j} = & \left( H_y \right)^{n}_{i, j} + \Delta t \left( \frac{\partial \hat{E}_z}{\partial x} \right)_{i, j}^n  \\ + & \left( c_x \left( \frac{\partial \hat{E}_z}{\partial x} \right)_{i, j}^n + b_x \left(\Psi_{H_y x}\right)_{i, j}^{n-1} \right) \Delta t.
\end{align*}
Here $\hat{H_x}$, $\hat{H}_y$ and $\hat{E}_z$ are corresponding linear approximation functions obtained by the least squares fitting.  

To apply BFECC to this scheme, we combine all the terms on the right hand side that involve spatial derivatives. For example, the equation for $E_z$ becomes
\begin{align*}
\left( E_z \right)^{n+1}_{i, j} = & \left( E_z \right)^{n}_{i, j} + \Delta t \left( (1 + c_x) \left( \frac{\partial \hat{H}_y}{\partial x} \right)_{i, j}^n - (1 + c_y) \left( \frac{\partial \hat{H}_x}{\partial y} \right)_{i, j}^n \right) \\ + & \left( b_x \left(\Psi_{E_z x}\right)_{i, j}^{n-1}  - b_y \left(\Psi_{E_z y}\right)_{i, j}^{n-1} \right) \Delta t.
\end{align*}
The term $\left( b_x \left(\Psi_{E_z x}\right)_{i, j}^{n-1}  - b_y \left(\Psi_{E_z y}\right)_{i, j}^{n-1} \right) \Delta t$ is treated as a source term. In the first two steps of the BFECC method, we ignore this source term. It is only there in the third step of BFECC. We can see that this requires very little modification to the scheme used in the computational domain.

Similarly, we can also use BFECC based on the least square $\theta$-scheme in the PMLs. 

\section{Numerical examples}\label{sec:num_examples}
\subsection{1D periodic solution}
We consider the following periodic initial condition for the 1D Maxwell's equations
\begin{align*}
E(0, x) = H(0, x) = \sin ( 2 \pi x). 
\end{align*}
The solution that satisfies the given initial condition is
\begin{align*}
E(t, x) = H(t, x) = \sin { 2 \pi (x + t) }.
\end{align*}

We solve the system with BFECC based on the central difference scheme from $t = 0$ to $t = 0.6$ with $\Delta t / \Delta x = 0.38$, $0.98$ and $1.7$, and compare the numerical solutions with the exact solution.

The order of accuracy result is summarized in Table-\ref{tb:1d_uniform_error_1}. The results confirm that BFECC based on the central difference scheme is second order accurate. Also note that the scheme is stable for $\Delta t = 1.7 \Delta x$, for which the classical Yee scheme becomes unstable. 

\begin{table}[H]
\centering
\caption{Order of accuracy for BFECC based on the central difference scheme at $T = 0.6$}\label{tb:1d_uniform_error_1}
\begin{tabular}{|c|c|c|c|c|c|c|} \hline
Grid & \multicolumn{2}{|c|}{$\Delta t/\Delta x = 0.38$} & \multicolumn{2}{|c|}{$\Delta t/\Delta x = 0.98$} & \multicolumn{2}{|c|}{$\Delta t/\Delta x = 1.7$} \\ \hline
 & Error & Order & Error & Order & Error & Order \\ \hline
$64$ & $1.11 \times 10^{-2}$ & -- & $2.50 \times 10^{-2}$ & -- & $5.58 \times 10^{-2}$ & -- \\ \hline
$128$ & $2.80 \times 10^{-3}$ & 2.00 & $6.41 \times 10^{-3}$ & 1.97 & $1.41 \times 10^{-2}$ & 1.99\\ \hline
$256$  & $7.93 \times 10^{-4}$  & 2.00 & $1.62 \times 10^{-3}$ & 2.00 & $3.58 \times 10^{-3}$ & 1.97\\ \hline
$512$  & $1.73 \times 10^{-4}$  & 2.00 & $4.00 \times 10^{-4}$ & 2.00 & $9.05 \times 10^{-4}$ & 1.99\\ \hline
$1024$  & $4.33 \times 10^{-5}$  & 2.00 & $1.00 \times 10^{-4}$ & 2.00 & $2.26 \times 10^{-4}$ & 2.00\\ \hline
$2048$  & $1.08 \times 10^{-5}$  & 2.00 & $2.51 \times 10^{-5}$ & 2.00 & $5.67 \times 10^{-5}$ & 2.00\\ 
\hline\end{tabular}
\end{table}

\subsection{2D periodic solution}
We consider the following periodic initial condition for the 2D Maxwell's equations in $\rm{TM}_{\rm{z}}$ mode.
\begin{align*}
& E_z(0, x, y) = \sin(2 \pi x) \\
& H_x(0, x, y) = 0 \\
& H_y(0, x, y) = -\sin(2 \pi x). 
\end{align*}
The exact solution is
\begin{align*}
& E_z(t, x, y) = \sin(2 \pi (x - t) ) \\
& H_x(0, x, y) = 0 \\
& H_y(0, x, y) = -\sin(2 \pi (x - t) ).
\end{align*}

We solve the system with BFECC based on the least square $\theta$-scheme from $t=0$ to $t=2.5$ with $\Delta t / \Delta x = 0.25$, and compare the solutions with the exact solutions. The problem is solved in four grids: (a) uniform rectangular grid; (b) non-rectangular grid obtained by a smooth perturbation from (a); (c) non-rectangular grid with a global circular grid deformation; and (d) non-rectangular grid with grid points shifted to a circular interface. The grids are shown in Figure-\ref{fig:periodic_4_grids} and the order of accuracy is shown in Table \ref{tb:2d_periodic_accuracy}. We see the numerical orders of accuracy are all above $2$, showing the effectiveness of the BFECC method on non-orthogonal grids. 

\begin{figure}[H]
\includegraphics[width=0.7\linewidth]{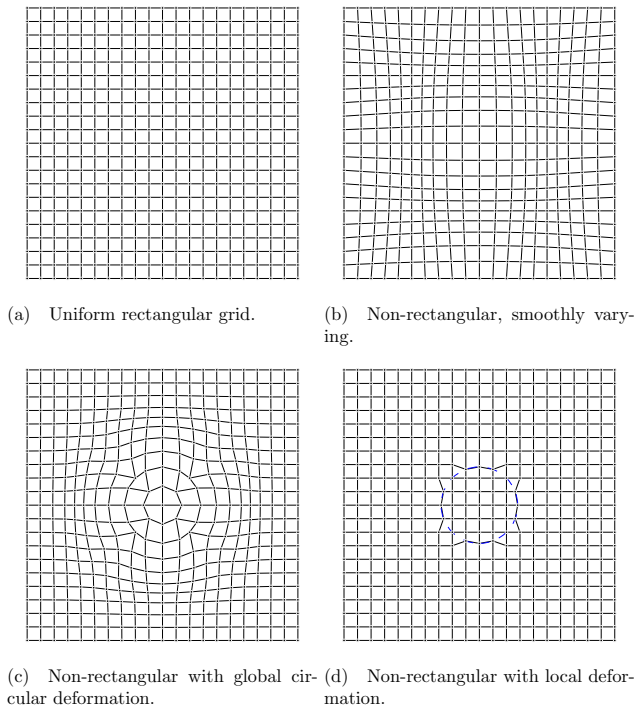}
\caption{Grids: (a) Uniform rectangular; (b) (c) and (d) Non-orthogonal grids.}
\label{fig:periodic_4_grids}
\end{figure}

\begin{table}[H]
\centering
\caption{Order of accuracy for BFECC based on the least square $\theta$-scheme at $T = 2.5$}\label{tb:2d_periodic_accuracy}
\resizebox{\columnwidth}{!}{%
\begin{tabular}{|c|c|c|c|c|c|c|c|c|} \hline
Grid & \multicolumn{2}{|c|}{(a)} & \multicolumn{2}{|c|}{(b)} & \multicolumn{2}{|c|}{(c)} & \multicolumn{2}{|c|}{(d)}  \\ \hline
 & Error & Order & Error & Order & Error & Order & Error & Order \\ \hline
$20 \times 20$ & $5.843 \times 10^{-2}$ & -- & $1.502 \times 10^{-1}$ & -- & $ 6.429 \times 10^{-2}$ & -- & $5.723 \times 10^{-2}$ & -- \\ \hline
$40 \times 40$ & $8.160 \times 10^{-3}$ & 2.84 & $2.469 \times 10^{-2}$ & 2.61 & $1.070 \times 10^{-2} $ & 2.59 & $7.013 \times 10^{-3}$ & 3.03 \\ \hline
$80 \times 80$  & $1.269 \times 10^{-3}$ & 2.69 & $3.426 \times 10^{-3} $ & 2.85 & $2.413 \times 10^{-3} $ & 2.15 & $8.485 \times 10^{-4}$ & 3.05\\ 
\hline\end{tabular}%
}
\end{table}

\subsection{Scattering by a dielectric cylinder}
In this example, we solve the 2D Maxwell's equations in $\rm{TM}_{\rm{z}}$ mode with BFECC based on the least square $\theta$-scheme for the scattering problem by a dielectric cylinder.
\begin{align*}
& \mu \frac{\partial H_x}{\partial t} = - \frac{\partial E_z}{\partial y} \\
& \mu \frac{\partial H_y}{\partial t} =  \frac{\partial E_z}{\partial x} \\
& \epsilon \frac{\partial E_z}{\partial t} = \frac{\partial H_y}{\partial x} - \frac{\partial H_x}{\partial y}.
\end{align*}

The incident wave is a $z$-polarized plane wave travelling in the $x$ direction, i.e. $ (E_z)_{inc} = \sin(\omega (x - t))$, $(H_x)_{inc} = 0$ and $(H_y)_{inc} = - \sin(\omega (x - t))$, where $\omega = 2 \pi / 0.6$ is the angular frequency. The computational domain is $[0, 1] \times [0, 1]$. A dielectric cylinder with $\epsilon = 2.25$ and $\mu = 1$ and radius $0.24$ is placed in the center of the computation domain. The surrounding medium has $\epsilon = 1$ and $\mu = 1$. Perfectly match layers are used as absorbing boundaries, and the total-field/scattered-field formulation \cite{taflove2005computational} is used to introduce plane waves into the computational domain.

Two grids are used in computation: (a) a uniform rectangular grid is used and the material interface is approximated by stair-casing; and (b) a point shifted grid in which intersection points of the uniform rectangular grid and the material interface are computed and the closest rectangular grid points are moved to the intersection points, and see Figure \ref{fig:point_shifted_grids} (a). We use a simple treatment for the material interface: if a grid point falls inside the dielectric cylinder, $\epsilon = 2.25$ and $\mu = 1$ are used during the update of $\bm{E}$ and $\bm{H}$, otherwise, $\epsilon = 1$ and $\mu = 1$ are used. Other interface treatments will be studied in the future.

BFECC based on the least square $\theta$-scheme is used instead of BFECC based on the least square central difference scheme is used. The larger numerical dissipation is helpful when there is material discontinuity. When BFECC based on the least square central difference scheme is used, there are small spurious oscillations presented in the numerical solution due to the material discontinuity. 

Since the CFL condition for BFECC based on the least square $\theta$-scheme only requires $\Delta t \leq \frac{\sqrt{3}}{\sqrt{(1/\Delta x)^2 + (1/\Delta y)^2}}$, here we take $\Delta t = \Delta x = \Delta y$. Smaller $\Delta t$ values have also been experimented, giving similar results as presented here. 

The numerical solution on the point-shifted grid at $t = 3.8$ is shown in Figure-\ref{fig:shift_mie_contour} and is compared with the analytic Mie solution \cite{bohren2008absorption} in Figure \ref{fig:shift_mie_slice}. BFECC based on the least square $\theta$-scheme scheme is able to generate smooth solutions without any spurious oscillation. $t = 3.8$ is chosen since the solution seems to reach the steady state at this time. The scheme has also been tested for several thousands time steps (up to $t = 12$) and the solution remains stable.

The grid refinement analysis for numerical solutions on uniform rectangular grids and point shifted grids is shown in Table \ref{tb:2d_mie_accuracy}. Here the numerical solution on a $320 \times 320$ grid is taken as the accurate solution, and all errors (in $l_2$) are computed with respect to this numerical solution. We can see that the BFECC scheme essentially achieves second order accuracy. 

\begin{table}[H]
\centering
\caption{Order of accuracy for BFECC based on the least square $\theta$-scheme at $T = 3.8$}\label{tb:2d_mie_accuracy}
\begin{tabular}{|c|c|c|c|c|} \hline
Grid & \multicolumn{2}{|c|}{uniform rectangular} & \multicolumn{2}{|c|}{non-rectangular} \\ \hline
& Error & Order & Error & Order  \\ \hline
$20 \times 20$ & $0.274$ & -- & 0.421 & --  \\ \hline
$40 \times 40$ & $0.0789$ & 1.80 & 0.130 & 1.69   \\ \hline
$80 \times 80$  & $0.0148$ & 2.41 & 0.0341 & 1.93  \\  \hline
$160 \times 160$  & $0.00370$ & 2.00 & 0.00683 & 2.33  \\ 
\hline\end{tabular}
\end{table}

\begin{figure}[H]
\includegraphics[width=0.9\linewidth]{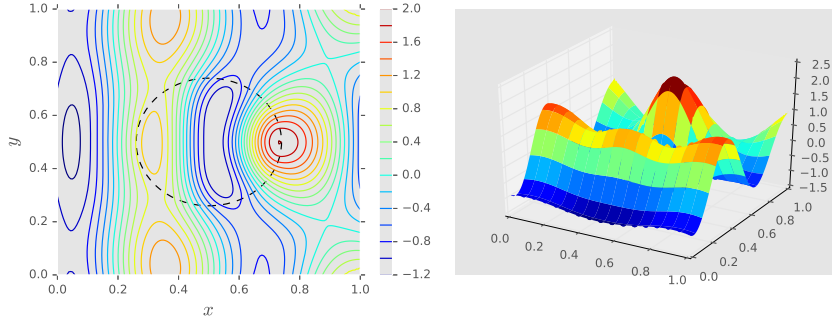}
\caption{BFECC based on the least square $\theta$-scheme solution at $t = 3.8$. Left: contour plot of $E_z$; Right: surface plot of $E_z$.}
\label{fig:shift_mie_contour}
\end{figure} 

\begin{figure}[H]
\includegraphics[width=0.7\linewidth]{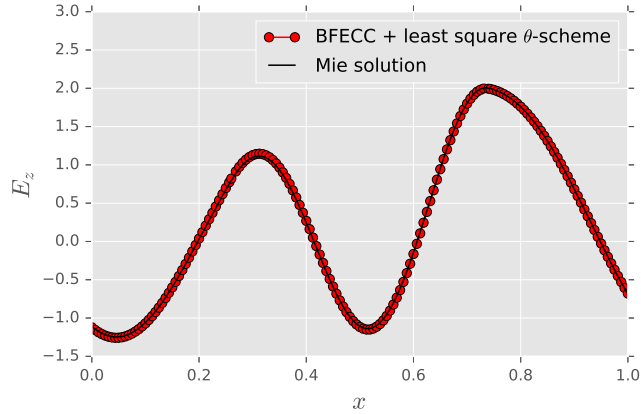}
\caption{Slice of $E_z$ with $y = 0.5$ at $t = 3.8$, compared with the analytic Mie solution.}
\label{fig:shift_mie_slice}
\end{figure} 

\subsection{Scattering by a dielectric object of complicated shape}
In this example, BFECC bases on the least square $\theta$-scheme is used to solve a scattering problem by a dielectric object of more complicated shape. The material setup is the same as in the previous example. Notice that the object has sharp corners and cavities inside, as shown in Figure \ref{fig:complex_shape_contour}. The two grids used for computations are (a) a uniform rectangular grid with staircasing approximation for material interface and (b) a point shifted grid, see Figure \ref{fig:point_shifted_grids} (b). 

The contour plot of $E_z$ at $t = 3.6$ is shown in Figure \ref{fig:complex_shape_contour}. Taking the $320 \times 320$ numerical solution as the reference, the numerical errors (in $l_2$) are shown in Table~\ref{tb:2d_complex_accuracy}. Again we see that the BFECC scheme is stable and has second order accuracy.

\begin{figure}[H]
\includegraphics[width=0.9\linewidth]{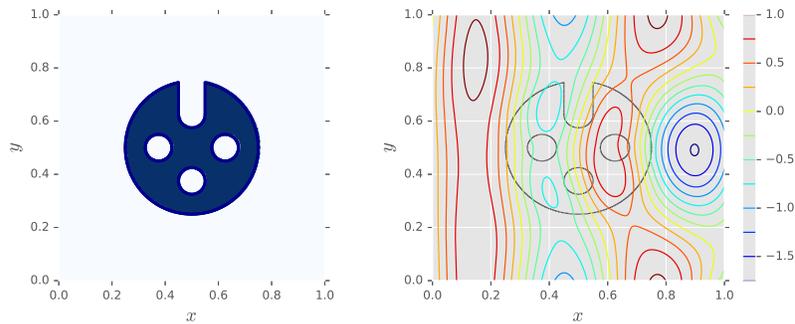}
\caption{Scattering by a complicated object. Left: shape of the object; Right: contour plot of $E_z$ at $t = 3.6$. }
\label{fig:complex_shape_contour}
\end{figure} 

\begin{table}[H]
\centering
\caption{Grid refinement analysis for BFECC based on the least square $\theta$-scheme at $T = 3.6$}\label{tb:2d_complex_accuracy}
\begin{tabular}{|c|c|c|c|c|} \hline
Grid & \multicolumn{2}{|c|}{uniform rectangular} & \multicolumn{2}{|c|}{non-rectangular} \\ \hline
 & Error & Order & Error & Order  \\ \hline
$20 \times 20$ & $4.200 \times 10^{-1}$ & -- & $4.384 \times 10^{-1}$ & --  \\ \hline
$40 \times 40$ & $1.159 \times 10^{-1}$ & 1.86 & $1.160 \times 10^{-1}$ & 1.92   \\ \hline
$80 \times 80$  & $3.618 \times 10^{-2}$ & 1.68 & $3.700 \times 10^{-2}$ & 1.65  \\  \hline
$160 \times 160$  & $8.116 \times 10^{-3}$ & 2.16 & $8.830 \times 10^{-3}$ & 2.07  \\ 
\hline\end{tabular}
\end{table}

\section{Conclusion}\label{sec:conclusion}
We study the Back and Forth Error Compensation and Correction (BFECC) Method for linear hyperbolic PDE systems and establish the stability and accuracy properties of BFECC for homogeneous linear hyperbolic systems with constant coefficients. The method is then applied to the Maxwell's equations. On uniform orthogonal grids, BFECC based on the central difference and BFECC based on the Lax-Friedrichs schemes are proved to be second order accurate and have larger CFL number than the classical Yee scheme. On non-orthogonal or unstructured grids, the BFECC method is applied to a first order scheme based on least square gradient approximation. Numerical examples demonstrate the effectiveness of the BFECC schemes for Maxwell's equations. In particular BFECC based on the least square central difference scheme or the least square $\theta$-scheme is easy to implement on non-orthogonal grids, has larger CFL numbers and second order accuracy in the numerical examples we have tested. We plan to test the BFECC schemes on unstructured grids with adaptive refinement for more complicated application problems in the future. 

\section{Acknowledgement}
The authors thank Jinjie Liu for helpful discussions on the Yee scheme.

\bibliographystyle{amsplain}

\bibliography{system_maxwell_bib}

\begin{appendix}
\section{Stability and accuracy of BFECC schemes based on central difference}\label{app:bfecc_CD_details}

\subsection{One dimensional case}

For Maxwell's equations in one dimensional free space with periodic boundary condition, central difference scheme $\mathcal{L}$'s Fourier symbol matrix is
\begin{align*}
Q_{\mathcal{L}} = 
\begin{pmatrix}
1 & i \lambda \sin( 2 \pi k h ) \\
i \lambda \sin( 2 \pi k h ) & 1
\end{pmatrix},
\end{align*}
and $\mathcal{L}^{*}$'s Fourier symbol matrix is $Q_{\mathcal{L}^{*}} = \overline{Q_{\mathcal{L}}}$. 

For the BFECC scheme based on central difference, its Fourier symbol matrix is 
\begin{align*}
Q_{B}  = Q_{\mathcal{L}} \left( I + \frac{1}{2} ( I - Q_{\mathcal{L}^{*}} Q_{\mathcal{L}} ) \right) = \left( 1 - \frac{1}{2} \lambda^2 \sin^2 ( 2 \pi k h ) \right) \begin{pmatrix}
1 & i \lambda \sin( 2 \pi k h ) \\
i \lambda \sin( 2 \pi k h ) & 1
\end{pmatrix}
\end{align*}

\emph{Stability}
We calculate eigenvalues for $Q_{\mathcal{L}}$ and $Q_{B}$
\begin{align*}
& \lambda(Q_{\mathcal{L}})_{\pm} = 1 \pm i \lambda \sin( 2 \pi k h) \\
& \lambda(Q_{B})_{\pm} = \left( 1 - \frac{1}{2} \lambda^2 \sin^2 ( 2 \pi k h ) \right) ( 1 \pm i \lambda \sin( 2 \pi k h))
\end{align*}
We study the spectral radius of $Q_{B}$:
\[
| \lambda(Q_{B})_{\pm} |^2 = \left( 1 - \frac{1}{2} \lambda^2 \sin^2 ( 2 \pi k h ) \right)^2 (1 + \lambda^2 \sin^2 ( 2 \pi k h ) )
\]
Let $\zeta = \sin^2( 2 \pi k h) \in [0, 1]$, and define
\[
f(\zeta) = | \lambda(Q_{B})_{\pm} |^2 = \left( 1 - \frac{1}{2} \lambda^2 \zeta \right)^2 ( 1 + \lambda^2 \zeta )
\]

When $\lambda^2 \leq 2$, $f(\zeta)$ is monotonically decreasing in $[0, 1]$, and it obtains its maximum at $0$, $f(0) = 1$ and for all $\zeta \in (0, 1]$, $f(\zeta) < 1$. For the case $f(0) = 1$, we can explicitly check that mode is stable. Therefore for $\lambda^2 \leq 2$, the scheme is stable. 

When $\lambda^2 > 2$, $\max_{\zeta \in [0, 1]} f(\zeta) = \max( f(0), f(1)) = \max\left( 1, \left(1 - \frac{1}{2} \lambda^2 \right)^2 ( 1 + \lambda^2 ) \right)$, we already checked $k = 0$ is always a stable mode. Setting $\left(1 - \frac{1}{2} \lambda^2 \right)^2 ( 1 + \lambda^2 ) < 1$, we get $\lambda^2 < 3$. 

Therefore $\Delta t / \Delta x = \lambda < \sqrt{3}$ ensures $l^2$ stability for the BFECC scheme. 

\emph{Accuracy}
Write 
\begin{align*}
E(t, x) = \sum_{k \in F_N} C_k(t) e^{2 \pi i k x} \\
H(t, x) = \sum_{k \in F_N} D_k(t) e^{2 \pi i k x} 
\end{align*}
and plug in the Maxwell's equations, we get:
\begin{align*}
\frac{d}{d t}
\begin{pmatrix}
C_k \\
D_k
\end{pmatrix} = \begin{pmatrix}
0 & 2 \pi i k \\
2 \pi i k & 0 
\end{pmatrix}
\begin{pmatrix}
C_k \\
D_k
\end{pmatrix} = G 
\begin{pmatrix}
C_k \\
D_k
\end{pmatrix}
\end{align*}
where matrix $G$ is defined by the last equality. Calculate the matrix exponential, we get: 
\begin{align*}
\begin{pmatrix}
C_k(t_n + \Delta t) \\
D_k(t_n + \Delta t) 
\end{pmatrix}
= e^{\Delta t G} 
\begin{pmatrix}
C_k(t_n) \\
D_k(t_n) 
\end{pmatrix}
= \begin{pmatrix}
\cos ( 2 \pi k \Delta t) &i \sin( 2 \pi k \Delta t) \\
i \sin ( 2 \pi k \Delta t) &\cos( 2 \pi k \Delta t)
\end{pmatrix}
\begin{pmatrix}
C_k(t_n) \\
D_k(t_n) 
\end{pmatrix}
\end{align*}

While with BFECC based on the central difference scheme, we have: 
\begin{align*}
\scalemath{0.8}{
\begin{pmatrix}
C_k(t_n + \Delta t) \\
D_k(t_n + \Delta t) 
\end{pmatrix}
= Q_{B} 
\begin{pmatrix}
C_k(t_n) \\
D_k(t_n) 
\end{pmatrix}
= 
\left( 1 - \frac{1}{2} \lambda^2 \sin^2 ( 2 \pi k h ) \right) \begin{pmatrix}
1 & i \lambda \sin( 2 \pi k h ) \\
i \lambda \sin( 2 \pi k h ) & 1
\end{pmatrix}
\begin{pmatrix}
C_k(t_n) \\
D_k(t_n) 
\end{pmatrix}}%
\end{align*}
Note $\lambda h = \Delta t$, we see: 
\begin{align*}
Q_{B} = e^{\Delta t G} + O( |k h|^{3} ), \, \text{as } h \rightarrow 0
\end{align*}

By the Theorem-\ref{thm:lax}, we see BFECC based on the central difference scheme is a second order accurate scheme.

Numerical dispersion relation can be obtained by noticing
\begin{align*}
e^{i \omega \Delta t} = \lambda\left( Q_{B} \right)
\end{align*}
where $\omega = 2 \pi \nu$ is the angular frequency and $\nu$ is the frequency. 
Taking the imaginary part, we get
\begin{align*}
\sin( \omega \Delta t) = \lambda \left( 1 - \frac{1}{2} \lambda^2 \sin^2 ( \tilde{k} h ) \right)  \sin ( \tilde{k}  h )
\end{align*}
where $\tilde{k} = 2 \pi k$. Therefore the numerical phase speed is
\begin{align*}
\frac{\omega}{\tilde{k}} = \frac{1}{\lambda \tilde{k} h } \arcsin\left[ \lambda \left( 1 - \frac{1}{2} \lambda^2 \sin^2 ( \tilde{k} h ) \right)  \sin ( \tilde{k}  h ) \right]
\end{align*}
Expand the right hand side upto second order, we get
\begin{align*}
\frac{\omega}{\tilde{k}} = \frac{1}{3} \lambda^2 \tilde{k}^2 h^2 - \frac{1}{6} \tilde{k}^2 h^2 + O(\tilde{k}^3 h^3)
\end{align*}

\subsection{Two dimensional case}
For Maxwell's equations in two dimensional free space with periodic boundary condition, central difference scheme $\mathcal{L}$'s Fourier symbol matrix is
\begin{align*}
Q_{\mathcal{L}} = \begin{pmatrix}
1 & 0 & - i \lambda_y \sin( 2 \pi l \Delta y) \\
0 & 1 &  i \lambda_x \sin( 2 \pi k \Delta x) \\
- i \lambda_y \sin( 2 \pi l \Delta y) & i \lambda_x \sin( 2 \pi k \Delta x) & 1
\end{pmatrix} 
\end{align*}
As discussed in Section \ref{subsec:2d}, $Q_{\mathcal{L}^{*}} = \overline{Q_{\mathcal{L}}}$. For convenience of notation, we denote $s^x_k = \sin(2 \pi k \Delta x)$ and $s^y_l = \sin(2 \pi l \Delta y)$. 

BFECC based on the central Difference scheme has Fourier symbol matrix
\begin{align*}
Q_B = Q_{\mathcal{L}} \left( I + \frac{1}{2} ( I - Q_{\mathcal{L}^{*}} Q_{\mathcal{L}}) \right)
\end{align*}
By direct computation, we get
\begin{align*}
I + \frac{1}{2} ( I - Q_{\mathcal{L}^{*}} Q_{\mathcal{L}}) = 
\begin{pmatrix}
1 - \frac{1}{2} \lambda_y^2 (s^y_l)^2  & \frac{1}{2} \lambda_x \lambda_y s^x_k s^y_l & 0 \\
\frac{1}{2} \lambda_x \lambda_y s^x_k s^y_l & 1 - \frac{1}{2} \lambda_x^2 (s^x_k)^2 & 0 \\
0 & 0 & 1 - \frac{1}{2} \lambda_x^2 (s^x_k)^2 - \frac{1}{2} \lambda_y^2 (s^y_l)^2
\end{pmatrix}
\end{align*}
and 
\begin{align}
Q_B = \scalemath{0.8}{
\begin{pmatrix}
1 - \frac{1}{2} \lambda_y^2 (s^y_l)^2  & \frac{1}{2} \lambda_x \lambda_y s^x_k s^y_l & - i \lambda_y s^y_l \left( 1 - \frac{1}{2} \lambda_x^2 (s^x_k)^2 - \frac{1}{2} \lambda_y^2 (s^y_l)^2 \right)   \\
\frac{1}{2} \lambda_x \lambda_y s^x_k s^y_l & 1 - \frac{1}{2} \lambda_x^2 (s^x_k)^2 &  i \lambda_x s^x_k \left( 1 - \frac{1}{2} \lambda_x^2 (s^x_k)^2 - \frac{1}{2} \lambda_y^2 (s^y_l)^2 \right)  \\
- i \lambda_y s^y_l \left( 1 - \frac{1}{2} \lambda_x^2 (s^x_k)^2 - \frac{1}{2} \lambda_y^2 (s^y_l)^2 \right) & i \lambda_x s^x_k \left( 1 - \frac{1}{2} \lambda_x^2 (s^x_k)^2 - \frac{1}{2} \lambda_y^2 (s^y_l)^2 \right) & 1 - \frac{1}{2} \lambda_x^2 (s^x_k)^2 - \frac{1}{2} \lambda_y^2 (s^y_l)^2
\end{pmatrix}
}
\end{align}

\emph{Stability}
Eigenvalues of $Q_{\mathcal{L}}$ are 
\begin{align*}
\lambda_1 = 1, \, \lambda_{2, 3} = 1 \pm i \sqrt{\lambda_x^2 (s^x_k)^2 + \lambda_y^2 (s^y_l)^2}
\end{align*}

The matrix $A$ can be decomposed as
\[
Q_{\mathcal{L}} = V \Lambda V^{-1}
\] 
where 
\begin{align*}
\Lambda = \text{diag} \{\lambda_1, \lambda_2, \lambda_3\}
\end{align*}
and 
\begin{align*}
V = 
\begin{pmatrix}
\lambda_x s^x_k & - \lambda_y s^y_l &\lambda_y s^y_l\\
\lambda_y s^y_l & \lambda_x s^x_k & - \lambda_x s^x_k \\
0 & \sqrt{\lambda_x^2 (s^x_k)^2 + \lambda_y^2 (s^y_l)^2 } & \sqrt{\lambda_x^2 (s^x_k)^2+ \lambda_y^2 (s^y_l)^2}
\end{pmatrix}
\end{align*}

It is then easy to see the scheme is $l^2$ stable if and only if $\max_{s^x_k, s^y_l} |\lambda_{2, 3}| \leq 1$, which is not true since $|\lambda_{2, 3}|^2 = 1 + \lambda_x^2 (s^x_k)^2 + \lambda_y (s^y_l)^2 > 1$ for $s^x_k \neq 0$ or $s^y_l \neq 0$. Therefore the central difference scheme is unconditionally unstable. 

We can verify that columns of $V$ are also eigenvectors of $\overline{Q_{\mathcal{L}}} Q_{\mathcal{L}}$ and hence eigenvectors of $Q_B$. This allows us to compute the eigenvalues of $Q_B$: 
\[
\lambda_1(Q_B) = 1, \lambda_{2, 3}(Q_B) = \left( 1 - \frac{1}{2} ( \lambda_x^2 (s^x_k)^2 + \lambda_y^2 (s^y_l)^2) \right) \left( 1 \pm i \sqrt{\lambda_x^2 (s^x_k)^2 + \lambda_y^2 (s^y_l)^2}\right)
\]
Therefore, BFECC based on the central difference scheme is stable if and only if
\begin{align*}
\max_{s^x_k, s^y_l} \left( 1 - \frac{1}{2} ( \lambda_x^2 (s^x_k)^2 + \lambda_y^2 (s^y_l)^2) \right)^2 \left( 1 + \lambda_x^2 (s^x_k)^2 + \lambda_y^2 (s^y_l)^2 \right) \leq 1
\end{align*}

Let $\zeta = (s^x_k)^2 \in [0, 1], \theta = (s^y_l)^2 \in [0, 1]$, define 
\[
f(\zeta, \theta) = \left( 1 - \frac{1}{2} ( \lambda_x^2 \zeta + \lambda_y^2 \theta) \right)^2 \left( 1 + \lambda_x^2 \zeta + \lambda_y^2 \theta \right)
\]

We have 
\begin{align*}
\frac{\partial f}{\partial \zeta} < 0 \\
\frac{\partial f}{\partial \theta} < 0
\end{align*}
when $1 - \frac{1}{2} ( \lambda_x^2 \zeta + \lambda_y^2 \theta) > 0$ and above this line, both partial derivatives are positive. 

Using this property, we see 
\begin{align*}
& \max_{0 \leq \zeta, \theta \leq 1} f(\zeta, \theta) = f(0, 0) = 1, \, \text{ if } \lambda_x^2 + \lambda_y^2 < 2 \\
& \max_{0 \leq \zeta, \theta \leq 1} f(\zeta, \theta) = \max (f(0, 0), f(1, 1)) , \, \text{ if } \lambda_x^2 + \lambda_y^2 > 2 
\end{align*}
For the case, $\lambda_x^2 + \lambda_y^2 > 2$ the $l^2$ stability condition becomes
\[
f(1, 1) = \left( 1 - \frac{1}{2} ( \lambda_x^2  + \lambda_y^2) \right)^2 \left( 1 + \lambda_x^2  + \lambda_y^2 \right) \leq 1 \Leftrightarrow \lambda_x^2  + \lambda_y^2 \leq 3
\]
Therefore, BFECC based on the central difference scheme is stable is stable if and only if 
\[
\lambda_x^2  + \lambda_y^2 \leq 3 \Leftrightarrow \Delta t \leq \frac{\sqrt{3}}{\sqrt{(1/\Delta x)^2 + (1/\Delta y)^2}}
\]

\emph{Accuracy}
Write
\begin{align*}
& H_x = \sum_{k, l \in \mathcal{F}_N} C_{k, l}(t) e^{2 \pi i ( k x + l y)} \\
& H_y = \sum_{k, l \in \mathcal{F}_N} D_{k, l}(t) e^{2 \pi i ( k x + l y)} \\
& E_z = \sum_{k, l \in \mathcal{F}_N} E_{k, l}(t) e^{2 \pi i ( k x + l y)}
\end{align*}
Plug into the Maxwell's equations and get
\begin{align*}
\frac{\partial }{\partial t} \begin{pmatrix}
C_{k, l} \\
D_{k, l} \\
E_{k, l}
\end{pmatrix} = 
\begin{pmatrix}
0 & 0 & - 2 \pi i l \\
0 & 0 & 2 \pi i k \\
- 2 \pi il & 2 \pi i k & 0
\end{pmatrix}
\begin{pmatrix}
C_{k, l} \\
D_{k, l} \\
E_{k, l}
\end{pmatrix} = G 
\begin{pmatrix}
C_{k, l} \\
D_{k, l} \\
E_{k, l}
\end{pmatrix}
\end{align*} 

Calculate the matrix exponential to get
\begin{align*}
\begin{pmatrix}
C_{k, l}( t + \Delta t) \\
D_{k, l}( t + \Delta t)\\
E_{k, l}( t + \Delta t)
\end{pmatrix} = e^{\Delta t G}
\begin{pmatrix}
C_{k, l}( t ) \\
D_{k, l}( t )\\
E_{k, l}( t )
\end{pmatrix}
\end{align*}
where
\begin{align*}
e^{\Delta t G} = 
\begin{pmatrix}
\frac{k^2 + l^2 \cos( 2 \pi \sqrt{k^2 + l^2} \Delta t)}{k^2 + l^2} & \frac{kl ( 1- \cos( 2 \pi \sqrt{k^2 + l^2} \Delta t))}{k^2 + l^2} & -i \frac{l \sin( 2 \pi \sqrt{k^2 + l^2} \Delta t) }{\sqrt{k^2 + l^2}} \\
\frac{kl ( 1- \cos( 2 \pi \sqrt{k^2 + l^2} \Delta t))}{k^2 + l^2} & \frac{l^2 + k^2 \cos( 2 \pi \sqrt{k^2 + l^2} \Delta t)}{k^2 + l^2} & i \frac{k \sin( 2 \pi \sqrt{k^2 + l^2} \Delta t) }{\sqrt{k^2 + l^2}} \\
-i \frac{l \sin( 2 \pi \sqrt{k^2 + l^2} \Delta t) }{\sqrt{k^2 + l^2}} & i \frac{k \sin( 2 \pi \sqrt{k^2 + l^2} \Delta t) }{\sqrt{k^2 + l^2}} & \cos( 2 \pi \sqrt{k^2 + l^2} \Delta t)
\end{pmatrix}
\end{align*}

Note it is symmetric. Expand entries of $e^{\Delta t G}$ upto second order, the entries are listed as (in the order of $(1, 1), (1, 2), (1, 3), (2, 2), (2, 3), (3, 3)$):
\begin{align*}
& \frac{k^2 + l^2 \cos( 2 \pi \sqrt{k^2 + l^2} \Delta t)}{k^2 + l^2} = 1 - \frac{1}{2} (2 \pi l^2) (\Delta t)^2 + O(\Delta t^3) \\
& \frac{kl ( 1- \cos( 2 \pi \sqrt{k^2 + l^2} \Delta t))}{k^2 + l^2} = \frac{1}{2} ( 2 \pi)^2 k l (\Delta t)^2 + O(\Delta t^3) \\
& -i \frac{l \sin( 2 \pi \sqrt{k^2 + l^2} \Delta t) }{\sqrt{k^2 + l^2}} = -i 2 \pi l \Delta t + O(\Delta t^3) \\
& \frac{l^2 + k^2 \cos( 2 \pi \sqrt{k^2 + l^2} \Delta t)}{k^2 + l^2} = 1 - \frac{1}{2} (2 \pi k^2) (\Delta t)^2 + O(\Delta t^3) \\
& i \frac{k \sin( 2 \pi \sqrt{k^2 + l^2} \Delta t) }{\sqrt{k^2 + l^2}} = i 2 \pi k \Delta t + O(\Delta t^3) \\
& \cos( 2 \pi \sqrt{k^2 + l^2} \Delta t) = 1 - \frac{1}{2} ( 2 \pi )^2 (k^2 + l^2) (\Delta t)^2 + O(\Delta t^3)
\end{align*}

Compare with entries of $Q_{\mathcal{L}}$, we see $Q_{\mathcal{L}} = e^{\Delta t G} + O(|\sqrt{k^2 + l^2} \Delta t|^2)$, by Theorem-\ref{thm:lax}, the central difference scheme is first order accurate. 

Expand entries of $Q_B$ and note $\lambda_x \Delta x = \Delta t$ and $\lambda_y \Delta y = \Delta t$, the entries are listed as (in the order of $(1, 1), (1, 2), (1, 3), (2, 2), (2, 3), (3, 3)$):
\begin{align*}
1 - \frac{1}{2}\lambda_y^2 (s^y_l)^2 & = 1 - \frac{1}{2} (2 \pi l^2) (\Delta t)^2 + O(\Delta t^3) \\
 \frac{1}{2} \lambda_x \lambda_y s^x_k s^y_l & = \frac{1}{2} ( 2 \pi)^2 k l (\Delta t)^2 + O(\Delta t^3) \\
 - i \lambda_y s^y_l \left( 1 - \frac{1}{2} (\lambda_x^2 (s^x_k)^2 + \lambda_y^2 (s^y_l)^2) \right) & = -i 2 \pi l \Delta t + O(\Delta t^3) \\
 1 - \frac{1}{2}\lambda_x^2 (s^x_k)^2 & = 1 - \frac{1}{2} (2 \pi k^2) (\Delta t)^2 + O(\Delta t^3) \\
 i \lambda_x s^x_k \left( 1 - \frac{1}{2} (\lambda_x^2 (s^x_k)^2 + \lambda_y^2 (s^y_l)^2) \right) & = i 2 \pi k \Delta t + O(\Delta t^3) \\
 1 - \frac{1}{2} (\lambda_x^2 (s^x_k)^2 + \lambda_y^2 (s^y_l)^2) & = 1 - \frac{1}{2} ( 2 \pi )^2 (k^2 + l^2) (\Delta t)^2 + O(\Delta t^3)
\end{align*}

Compare with entries of $e^{\Delta t G}$, we see $Q_B = e^{\Delta t G} + O(|\sqrt{k^2 + l^2} \Delta t|^3)$, by Theorem-\ref{thm:lax}, BFECC based on the central difference scheme is second order accurate.
\end{appendix}

\end{document}